\documentclass[acmsmall]{acmart}

\usepackage[acmart]{lindquist}
\usepackage{lu-abbrev}

\usepackage{algpseudocode,algorithm,algorithmicx}
\usepackage{booktabs}
\usepackage{microtype}

\usepackage{pgfplots}
\pgfplotsset{compat=1.17}

\iftrue
\usetikzlibrary{external}
\fi

\usepackage[capitalize]{cleveref}
\crefname{equation}{}{}
\crefformat{footnote}{#2\footnotemark[#1]#3}
\makeatletter
\newcommand{\crefTheoremName}{\cref@theorem@name}
\newcommand{\crefTheoremNamePlural}{\cref@theorem@name@plural}
\newcommand{\crefLemmaName}{\cref@lemma@name}
\makeatother

\newtheorem{theorem}{Theorem}
\newtheorem{lemma}[theorem]{Lemma}

\usepackage{ifthen}
\DeclarePairedDelimiter\btsize{\langle}{\rangle}
\newcommand{\rbt}[3][]{\ifthenelse{\equal{#2}{}}{\mathcal{#3}}{\mathcal{#3}^{\btsize{#2}#1}}}
\newcommand{\matname}[1]{\texttt{#1}}

\newcommand{\seq}[2]{S^{#1}_{#2}}
\newcommand{\conseq}[2]{CS^{#1}_{#2}}
\newcommand{\seqeq}[2]{#1\cong#2}
\newcommand{\seqeqmod}[3]{#1\equiv#2\pmod{#3}}
\newcommand{\seqneqmod}[3]{#1\not\equiv#2\pmod{#3}}

\title[Generalizing Random Butterfly Transforms to Arbitrary Matrix Sizes]{Generalizing Random Butterfly Transforms\texorpdfstring{\\}{ }to Arbitrary Matrix Sizes}
\author{Neil Lindquist}
\orcid{0000-0001-9404-3121}
\affiliation{%
	\institution{The University of Tennessee}
	\streetaddress{Suite 203 Claxton, 1122 Volunteer Blvd}
	\city{Knoxville}
	\state{TN}
	\postcode{37996}
	\country{USA}
}
\email{nlindqu1@icl.utk.edu}
\author{Piotr Luszczek}
\orcid{0000-0002-0089-6965}
\affiliation{%
	\institution{The University of Tennessee}
	\streetaddress{Suite 203 Claxton, 1122 Volunteer Blvd}
	\city{Knoxville}
	\state{TN}
	\postcode{37996}
	\country{USA}
}
\email{luszczek@icl.utk.edu}
\author{Jack Dongarra}
\orcid{0000-0003-3247-1782}
\affiliation{%
	\institution{The University of Tennessee}
	\streetaddress{Suite 203 Claxton, 1122 Volunteer Blvd}
	\city{Knoxville}
	\state{TN}
	\postcode{37996}
	\country{USA}
}
\email{dongarra@icl.utk.edu}

\acmJournal{TOMS}



\begin{document}

\begin{abstract}
Parker and L\^e introduced \acfp{rbt} as a preprocessing technique to replace pivoting in dense LU factorization.
Unfortunately, their \acs{fft}-like recursive structure restricts the
dimensions of the matrix.
Furthermore, on multi-node systems, efficient management of the communication overheads
restricts the matrix's distribution even more.
To remove these limitations, we have generalized the \acs{rbt} to arbitrary matrix sizes by truncating the dimensions of each layer in the transform.
We expanded Parker's theoretical analysis to generalized \acs{rbt},
specifically that in exact arithmetic, \acl{genp} will succeed with probability~1 after transforming a matrix with full-depth \acsp{rbt}.
Furthermore, we experimentally show that these generalized transforms improve performance over Parker's formulation by up to \SI{62}{\%} while retaining the ability to replace pivoting.
This generalized \acs{rbt} is available in the \acs{slate} numerical
software library.
\end{abstract}
\maketitle
\acbarrier

\section{Introduction}

\Ac{gepp} is commonly used to solve large, dense systems of linear
equations.
For numerical stability, the computational elimination is interleaved
with the row exchanges to maximize the magnitude of the diagonal elements,
which subsequently scale the outer-product updates.
However, most pivoting procedures incur significant
communication overheads due to the growing performance
imbalance between computational capacity and communication bandwidth
available on modern supercomputers.
Furthermore, pivoting adds data dependencies that limit the available
parallelism, not only asymptotically but also in most practical settings.
Unfortunately, \ac{genp} is numerically unstable for most linear systems
originating in scientific applications.
An alternative to pivoting is preprocessing the matrix with a random
transform before factoring it with \ac{genp};
such randomization provides numerical stability by preventing large element growth.
A popular choice for this preprocessing is Parker and L\^e's \ac{rbt}~\cite{parkerHowEliminatePivoting1995}.

\Ac{rbt} preprocessing is similar to combined left- and
right-preconditioning, except the ``preconditioned'' matrix is
explicitly computed so that it can be factored.
In other words, for a pair of transforms, \(\rbt{}{U}^T\) and \(\rbt{}{V}\), the linear system \(Ax=b\) is rewritten as
\begin{equation}
\label{eq:precond}
	(\rbt{}{U}^T\! A \rbt{}{V})(\rbt{}{V}^{-1} x) = (\rbt{}{U}^T b).
\end{equation}
(While we use \(\rbt{}{U}^T\) here,
 \(\rbt{}{U}^*\) can also be used in the
complex case.)
From \cref{eq:precond}, the algorithm falls out naturally as
follows:
\begin{enumerate}
\item Transform the system 
      as \( \widetilde{A}\equiv\rbt{}{U}^T A \rbt{}{V} \)
      and \(\widetilde{b} \equiv \rbt{}{U}^T b\).
\item Solve \(\widetilde{A}\widetilde{x}=\widetilde{b}\) without pivoting.
\item Transform the solution by \(\rbt{}{V}\) to undo the implicit \(\rbt{}{V}^{-1}\).
\end{enumerate}
This approach requires fast transforms to avoid adding a significant overhead to the factorization.
The \ac{rbt} has an \acs{fft}-like structure which limits the cost of
transforming \(A \in \mathbb{R}^{n\times n}\) to \(\Oh{n^2\log_2(n)}\).
Furthermore, the transform is often truncated to reduce the logarithmic term to a small constant.


Unfortunately, \acp{rbt} are limited to matrix sizes that are multiples of \(2^d\) (where \(d\) is the depth) due to the \acs{fft}-like structure.
Furthermore, efficient application in distributed settings requires the matrix distribution to be aligned to the butterfly structure~\cite{baboulinEfficientDistributedRandomized2014,lindquistReplacingPivotingDistributed2020}.
Thus, applications must pad their matrices with the identity matrix to fit the transform, requiring more operations in the factorization and a matrix allocation that depends on the transform's depth.
However, the preprocessing of linear systems does not depend on the
exact numerical properties of the chosen \ac{rbt} but simply on its
ability to scramble the elements from all over the original matrix using weighted sums.
Thus, we propose a generalized structure for the \ac{rbt} that allows arbitrarily sized \acp{rbt} and efficient distributed execution.
Our Generalized \ac{rbt} takes the previously used
formulation of \ac{rbt} (herein called a \emph{Parker \ac{rbt}}) and truncates \emph{each layer separately} to the desired size.

\section{Previous Work}

The \ac{rbt} approach was first proposed in 1995 through a pair of tech reports by Parker~\cite{parkerRandomButterflyTransformations1995,parkerRandomizingButterflyTransformation1995} and a third tech report by Parker and L\^e~\cite{parkerHowEliminatePivoting1995}.
They outline the approach and provide basic theoretical and experimental analysis.
Their theoretical analysis shows that with probability 1, \ac{genp} will have only non-zero pivots in exact arithmetic.
They then tested 11 matrices with sizes from \(n=32\) to \(n=512\).
The \ac{rbt} solver provided similar solutions to LINPACK for most of the matrices; however, its solutions for ill-conditioned problems were less accurate than LINPACK.
Due to a lack of optimizations, their implementation performed worse than LINPACK.

Then between 2008 and 2016, Baboulin et al.\ refined the \ac{rbt} idea into a performant solver over numerous papers.
The primary thrust of their work targeted many-core and heterogeneous systems for both non-symmetric and symmetric-indefinite problems~\cite{baboulinIssuesDenseLinear2008,tomovDenseLinearAlgebra2010,beckerReducingAmountPivoting2012,baboulinAcceleratingLinearSystem2013,baboulinRandomizedLUbasedSolver2015,baboulinDenseSymmetricIndefinite2016}.
This line of work included several improvements, including
\begin{itemize}
	\item truncating the transform's recursion depth to two,
	\item designing efficient \ac{rbt} kernels, and
	\item using iterative refinement.
\end{itemize}
Overall, their work shows excellent speedups over partial pivoting and reliably solved the test matrices, although the accuracy tests were limited to problems of size \(n=1024\).
Additionally, they explored the use of the \ac{rbt} strategy for a distributed, symmetric-indefinite solver~\cite{baboulinEfficientDistributedRandomized2014}, for sparse factorization~\cite{baboulinUsingRandomButterfly2015}, and for incomplete sparse factorizations~\cite{baboulinUsingRandomButterfly2015a,jamalHybridCPUGPU2016}.

Building on the work of Baboulin et al., Donfack et al.\ compared
different replacements for partial pivoting, (including the \ac{rbt}),
in a single-node, multicore setting~\cite{donfackSurveyRecentDevelopments2015}.
In that work, the speedup of the \ac{rbt} solver compared to partial pivoting was much lower than in the GPU-based studies, likely because CPUs handle irregular work better than GPUs.
But, more notably, they tested the accuracy of \ac{rbt} with \(n=\num{30000}\) and demonstrated that a depth-2 \ac{rbt} fails to sufficiently transform all problems (specifically the \matname{ris} and \matname{orthog} matrices).
We also previously extended this style of \ac{rbt} to a distributed, heterogeneous LU-factorization~\cite{lindquistReplacingPivotingDistributed2020}.
Our implementation demonstrated large speedups over partial pivoting and accurate results for a few matrices of size \(n=\num{100000}\).
However, like Donfack et al.\, we found that the \ac{rbt} is ineffective on large versions of the \matname{orthog} matrix.
(See \cref{sec:exp:orr} for analysis of using the \ac{rbt} on \matname{ris} and \matname{orthog}.)

An interesting, recent work by Shen et al.\ combines \acp{rbt} with a modified version of the \ac{aca} algorithm~\cite{shenTaskparallelTiledDirect2022}.
The \ac{aca} algorithm takes advantage of low-rank properties in the
matrix's off-diagonal submatrices to factor a dense matrix in
\(\Oh{n\log(n)}\) time.
However, intuition suggests these optimizations would work against each
other since the \ac{rbt} tries to equalize the rank between the on- and
off-diagonal parts of the matrix while \ac{aca} relies on the
off-diagonal submatrices being low-rank;
this is likely the cause of their experimental loss of accuracy compared to the regular \ac{rbt} solver.
Unfortunately, their experiments are limited to matrices that can be factored without pivoting, and they did not test the performance of a non-\ac{rbt}, non-pivoted factorization (with or without \ac{aca}), making it unclear whether the randomization actually benefited the accuracy.

Additionally, there has been recent work on the theoretical properties of \acp{rbt}~\cite{trogdonSpectralNumericalProperties2019,peca-medlinGrowthFactorsRandom2023}.
Those results show that preprocessing the identity matrix with \acp{rbt} results in a quadratic median growth for \ac{genp} (suggesting that \acp{rbt} will not significantly increase the growth for a given matrix).
Furthermore, they experimentally show that, for \(n=256\), the majority of \acp{rbt} reduce the growth of \ac{genp} applied to Wilkinson's matrix~\cite{wilkinsonAlgebraicEigenvalueProblem1965}  from \(10^{77}\) to less than \(10^6\).
Unfortunately, they considered butterflies based on rotation matrices instead of the butterflies based on block-Hadamard matrices used by Baboulin et al.

Butterfly matrices are not the only transform that has been proposed for randomized preprocessing.
First, Parker and Pierce proposed the ``Randomizing Fast Fourier Transform'', which combines an \ac{fft} with a random, diagonal scaling, as an alternative to the \ac{rbt}~\cite{parkerRandomizingFFTAlternative1995}.
This is related to later sub-sampling transforms which randomly sample the columns of a randomizing fast Fourier transform (although, later transforms are usually formulated as the transpose).
Such transforms include the fast Johnson-Lindenstrauss transform~\cite{ailonApproximateNearestNeighbors2006}, the \acl{srft}~\cite{woolfeFastRandomizedAlgorithm2008}, and the subsampled random Hadamard transform~\cite{nguyenFastEfficientAlgorithm2009}.
They proved that applying a randomizing fast Fourier transform to either side of a matrix resulted in a strongly non-singular matrix with probability 1.
Pan et al.\ have also explored numerous transforms, including Gaussian and circulant matrices~\cite{panRandomMultipliersNumerically2015,panNumericallySafeGaussian2017}.
They proved that a matrix preconditioned with one or more Gaussian matrices can be factored with a limited growth factor using \ac{genp} with high probability.
Furthermore, experimental results with various preconditioning matrices show numerical errors of \(10^{-10}\) to \(10^{-15}\) with \ac{genp}, one step of iterative refinement, and matrices of size \(n=256\) to \(n=4096\).

Finally, there exist other types of randomized preprocessing.
A related idea to the random multipliers is to instead add a random matrix, which can be corrected with the Woodbury formula~\cite{panRandomizedPreprocessingPivoting2013}.
For low-rank problems, ``sketching'' techniques have recently become
popular~\cite{martinssonRandomizedNumericalLinear2020} and are advancing
rapidly~\cite[Sec.~2, 7, and A]{murrayRandomizedNumericalLinear2023}.
These techniques multiply the system matrix by a random, rectangular matrix to reduce the problem to a smaller dimension, while retaining the large singular values.

Besides randomized preprocessing, several other strategies
address the cost of pivoting in dense Gaussian elimination.
Most notably, \ac{getp} selects a whole block of pivots at a time to reduce the number of synchronizations needed~\cite{grigoriCALUCommunicationOptimal2011}.
Threshold pivoting also modifies the pivot selection process to allow smaller pivots that require less data movement~\cite{lindquistThresholdPivotingDense2022}.
Another strategy is dynamic pivoting, which avoids swapping rows in memory when pivoting and instead rearranges the matrix distribution~\cite{geistLUFactorizationAlgorithms1988}.
Finally, \ac{beam} replaces pivoting with additive perturbations to
diagonal blocks during the
factorization~\cite{lindquistUsingAdditiveModifications2023}; this
reduces element growth in the Schur-complements but requires a correction with either
iterative refinement or the Woodbury formula.


\section{Generalized Random Butterfly Transforms}
\label{sec:rbt}
\Aclp{rbt} are constructed by alternating random, diagonal matrices with orthogonal matrices.
Our generalized formulation uses orthogonal butterfly matrices of the form
\begin{equation}
	\label{eq:O-matrix-definition}
	O^{\btsize{m}}_{\mu,\nu} = \begin{bmatrix} \tfrac{1}{\sqrt{2}}I_\mu & 0 & \tfrac{1}{\sqrt{2}}I_\mu \\ 0 & I_\nu & 0 \\ \tfrac{1}{\sqrt{2}}I_\mu & 0 & \tfrac{-1}{\sqrt{2}}I_\mu \end{bmatrix}
\end{equation}
where \(2\mu + \nu = m\) is the matrix dimension
and \(I_{\mu}\) is an identity matrix of size \(\mu\).
We call butterflies where \(\nu = 0\) or \(m=1\) Parker butterflies since they match Parker's original structure~\cite{parkerRandomButterflyTransformations1995}.
Interestingly, \cref{fig:butterfly-shape} shows that these generalized butterflies still have a data dependency diagram that looks like a butterfly; the generalization merely adds the butterfly's body.
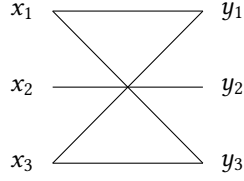
\begin{figure}
	\centering
	\begin{tikzpicture}
	\draw (0, -1) -- (2, -1);
	\draw (0, -1) -- (2,  1);
	\draw (0,  0) -- (2,  0);
	\draw (0,  1) -- (2, -1);
	\draw (0,  1) -- (2,  1);

	\node at (-0.4, -1) {\(x_3\)};
	\node at (-0.4,  0) {\(x_2\)};
	\node at (-0.4,  1) {\(x_1\)};
	\node at ( 2.4, -1) {\(y_3\)};
	\node at ( 2.4,  0) {\(y_2\)};
	\node at ( 2.4,  1) {\(y_1\)};
\end{tikzpicture}
	\caption{Data dependencies for multiplying the vector \([x_1^T, x_2^T, x_3^T]^T\) by an orthogonal butterfly matrix to produce \([y_1^T, y_2^T, y_3^T]^T\!\).}
	\Description{Data dependencies diagram with \(x_1\) going to \(y_1\) and \(y_3\), \(x_2\) going to \(y_2\), and \(x_3\) going to \(y_1\) and \(y_3\).  This looks like a pictogram of a butterfly.}
	\label{fig:butterfly-shape}
\end{figure}
These orthogonal matrices can be composed with random diagonal matrices into a depth-\(d\) \ac{rbt}:
\begin{equation}
	\label{eq:rbt-depth-d}
	\rbt{n}{U} = \mathrm{diag}\bigl(O^{\btsize{\ell_{d,1}}}_{\mu_{d,1},\nu_{d,1}}, O^{\btsize{\ell_{d,2}}}_{\mu_{d,2},\nu_{d,2}}, \dots\bigr) R_d \cdots O^{\btsize{n}}_{\mu_{1,1},\nu_{1,1}}\!R_1
\end{equation}
where the sizes are chosen such that pairs of butterflies from one layer align with a single butterfly from the layer to the right.
Specifically, if \(\ell_{i,j}\) is the size of the \(j\)th butterfly (top to bottom) in the \(i\)th layer (right to left), then
\begin{equation}
	\label{eq:butterfly-alignment}
	\ell_{i+1,2j} + \ell_{i+1,2j+1} = \ell_{i,j}, \qquad
	\ell_{i+1,2j} = \mu_{ij}+\nu_{ij}, \qquad \text{and} \qquad
	\ell_{i+1,2j+1} = \mu_{ij}.
\end{equation}
\Cref{fig:rbt-structure-diagram} demonstrates this structure.
\begin{figure*}
	\centering
	\begin{tikzpicture}[scale=3.2,
					nlborder/.style={thin,black},
					nlnonzeros/.style={thick,black},
					nldivider/.style={dotted,black}]

	\node at (1.25, -0.5) {\(R_1\)};
	\node at (1.10, -0.5) {\(\times\)};

	\draw[nlborder] (0,0) -- (0,-1) -- (1,-1) -- (1,0) -- cycle;
		\draw[nldivider] (0.0,-0.4) -- (1.0,-0.4);
		\draw[nldivider] (0.0,-0.6) -- (1.0,-0.6);
		\draw[nldivider] (0.4,-0.0) -- (0.4,-1.0);
		\draw[nldivider] (0.6,-0.0) -- (0.6,-1.0);

		\draw[nlnonzeros] (0.0,-0.0) -- (1.0,-1.0);
		\draw[nlnonzeros] (0.0,-0.6) -- (0.4,-1.0);
		\draw[nlnonzeros] (0.6,-0.0) -- (1.0,-0.4);

	\node at (-0.10, -0.5) {\(\times\)};
	\node at (-0.25, -0.5) {\(R_2\)};
	\node at (-0.40, -0.5) {\(\times\)};

	\begin{scope}[xshift=0.1cm]
		\draw[nlborder] (-1.6,0) -- (-1.6,-1) -- (-0.6,-1) -- (-0.6,0) -- cycle;
			\draw[nldivider] (-1.6,-0.6) -- (-0.6,-0.6);
			\draw[nldivider] (-1.0,-0.0) -- (-1.0,-1.0);
			\draw[nldivider] (-1.6,-0.3) -- (-1.0,-0.3);
			\draw[nldivider] (-1.3,-0.0) -- (-1.3,-0.6);
			\draw[nldivider] (-1.0,-0.7) -- (-0.6,-0.7);
			\draw[nldivider] (-1.0,-0.9) -- (-0.6,-0.9);
			\draw[nldivider] (-0.9,-0.6) -- (-0.9,-1.0);
			\draw[nldivider] (-0.7,-0.6) -- (-0.7,-1.0);

			\draw[nlnonzeros] (-1.6,-0.0) -- (-0.6,-1.0);
			\draw[nlnonzeros] (-1.3,-0.0) -- (-1.0,-0.3);
			\draw[nlnonzeros] (-1.6,-0.3) -- (-1.3,-0.6);
			\draw[nlnonzeros] (-1.0,-0.9) -- (-0.9,-1.0);
			\draw[nlnonzeros] (-0.7,-0.6) -- (-0.6,-0.7);
	\end{scope}

	\node at (-1.60, -0.5) {\(\times\)};
	\node at (-1.75, -0.5) {\(R_3\)};
	\node at (-1.90, -0.5) {\(\times\)};

	\begin{scope}[xshift=-3cm]
		\draw[nlborder] (0,0) -- (0,-1) -- (1,-1) -- (1,0) -- cycle;
			\draw[nldivider] (0.0,-0.3) -- (0.6,-0.3);
			\draw[nldivider] (0.0,-0.6) -- (1.0,-0.6);
			\draw[nldivider] (0.6,-0.9) -- (1.0,-0.9);
			\draw[nldivider] (0.3,-0.0) -- (0.3,-0.6);
			\draw[nldivider] (0.6,-0.0) -- (0.6,-1.0);
			\draw[nldivider] (0.9,-0.6) -- (0.9,-1.0);
			\draw[nldivider] (0.00,-0.15) -- (0.30,-0.15);
			\draw[nldivider] (0.15,-0.00) -- (0.15,-0.30);
			\draw[nldivider] (0.30,-0.45) -- (0.60,-0.45);
			\draw[nldivider] (0.45,-0.30) -- (0.45,-0.60);
			\draw[nldivider] (0.60,-0.75) -- (0.90,-0.75);
			\draw[nldivider] (0.75,-0.60) -- (0.75,-0.90);

			\draw[nlnonzeros] (0.0,-0.0) -- (1.0,-1.0);
			\draw[nlnonzeros] (0.15,-0.00) -- (0.30,-0.15);
			\draw[nlnonzeros] (0.00,-0.15) -- (0.15,-0.30);
			\draw[nlnonzeros] (0.45,-0.30) -- (0.60,-0.45);
			\draw[nlnonzeros] (0.30,-0.45) -- (0.45,-0.60);
			\draw[nlnonzeros] (0.75,-0.60) -- (0.90,-0.75);
			\draw[nlnonzeros] (0.60,-0.75) -- (0.75,-0.90);
	\end{scope}
\end{tikzpicture}
	\caption{The structure of a depth-3, semi-Parker \ac{rbt}.}
	\label{fig:rbt-structure-diagram}
	\Description{Sparsity diagrams for the three orthogonal matrices in a depth-3 \ac{rbt}.  The last butterfly on each diagonal is truncated, but submatrices of one layer align nicely with those of the next.}
\end{figure*}
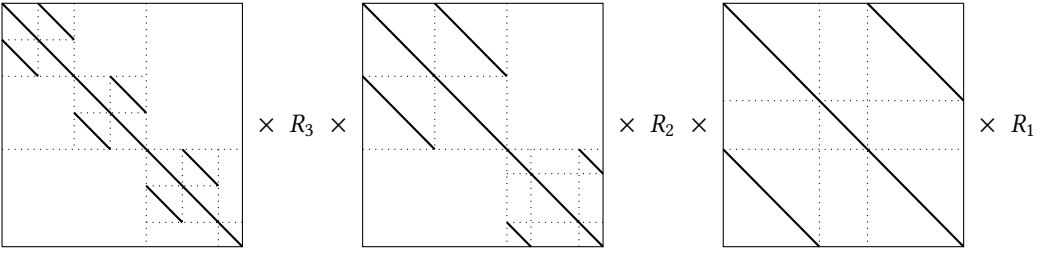
Note how the nice recursive structure of the \ac{rbt} is retained while
disconnecting the dimensions of the \ac{rbt} layers from the global
problem size.
This removes the need to pad the matrix to a particular size.
Furthermore, in the distributed case the sizes of the transforms can be aligned to the matrix distribution,
which avoids complicated element-wise communication patterns~\cite{baboulinEfficientDistributedRandomized2014,lindquistReplacingPivotingDistributed2020}.

Previous formulations of the \ac{rbt} (which we call \emph{Parker \acp{rbt}}) are equivalent to \cref{eq:rbt-depth-d} when all the individual butterflies are Parker butterflies.
Because Parker \acp{rbt} have a track record of providing accurate results,
we ideally want to deviate from Parker \acp{rbt} by the minimum amount
needed to obtain efficient execution.
So, we define a \emph{semi-Parker \ac{rbt}} to be an \ac{rbt} such that all the butterflies in \cref{eq:rbt-depth-d} are Parker butterflies except for the last block in each block-diagonal matrix.
In other words, a semi-Parker \ac{rbt} is like a Parker \ac{rbt}, except each of the orthogonal layers is truncated to match the size of the system matrix.
(However, a semi-Parker \ac{rbt} is \emph{not} equivalent to a Parker \ac{rbt} that has been truncated after multiplying all the constituent matrices.)
We focus on semi-Parker \acp{rbt}, but non-semi-Parker \acp{rbt} would be useful for a matrix distributed with non-uniform tile sizes.

Like previous computational-focused works, we formulate our butterflies using random diagonal matrices and (truncated) block-Hadamard matrices.
However, the theoretical analysis of Trogdon and Peca-Medlin uses butterflies based on block-rotation matrices.
The idea of this generalization also applies to those matrices: simply use a rotation of \ang{0} for the truncated portion.
In other words, the individual butterfly matrices would take the form
\[
	\begin{bmatrix}
		C_\mu & 0 & S_\mu \\
		0 & I_\nu & 0 \\
		S_\mu & 0 & -C_\mu \\
	\end{bmatrix}
\]
with \(C,S\) being diagonal matrices such that \(C_\mu^2 + S_\mu^2 = I_\mu\).
Then, these butterfly matrices would be composed as in \cref{eq:rbt-depth-d}, except with \(R_i=I\).
Such \acp{rbt} have the benefit of being orthogonal.


Generalized \acp{rbt} are implemented similarly to Parker \acp{rbt},
except the 1-element rows of the transforms simply require scaling
by the random factor.
For a 2-sided kernel, the loops must be split into four behaviors (each side can have either 1 or 2 elements).
So, similar to previous formulations of the \ac{rbt},
a depth-\(d\) transform can be applied to a vector in \(2dn\) \si{\flop} and to both sides of a matrix in \(4dn^2\) \si{\flop}
when the normalization factors of \cref{eq:rbt-depth-d} are combined into the diagonal matrices.
Furthermore, the storage of an \ac{rbt} also remains at \(dn\) words.

\Cref{alg:implementation} presents the outline of an \ac{fft}-like application of a two-sided semi-Parker butterfly.
\begin{algorithm*}[p]
	\caption{Two-sided semi-Parker \ac{rbt} application \(\rbt{}{U}^T\!A\rbt{}{V}\) using packed butterfly storage.}
	\label{alg:implementation}
	\begin{algorithmic}[1]
	\Procedure{RBT}{\(A\in\mathbb{R}^{n\times n}\), \(U\in\mathbb{R}^{n\times d}\), \(V\in\mathbb{R}^{n\times d}\){}}
		\State \(d \gets\) \ac{rbt} depth
		\State \(m \gets\) round up the dimension of \(A\)
			such that \(2^{-d}m\) is nicely aligned
			\label{alg:implementation:reference-size}
		\For{\(k\) from \(d-1\) to \(0\)} \label{fig:implementation:loop-1}
			\State \(b_n \gets 2^k\) \Comment{Number of butterflies}
			\State \(h \gets m/(2b_n)\) \Comment{Size of half of a butterfly}

			\For{\(b_i,b_{\!j} \in \{0,\ldots,b_n\}^2\)}
				\State \(j_1 \gets 2b_{\!j}h;\  {\!j}_2 \gets {\!j}_1+h;\  {\!j}_3 \gets \min({\!j}_2+h, n)\) \Comment{Column indices for right butterflies}
				\State \(i_1 \gets 2b_ih;\  i_2 \gets i_1+h;\  i_3 \gets \min(i_2+h, n)\) \Comment{Row indices for left butterflies}

				\State \(\text{\textproc{RBT2}}(A[\range{i_1}{i_2}, \range{j_1}{j_2}], A[\range{i_1}{i_2}, \range{j_1}{j_2}], A[\range{i_2}{i_3}, \range{j_2}{j_3}], A[\range{i_2}{i_3}, \range{j_2}{j_3}],\)
				\Statex \hskip\algorithmicindent\hskip\algorithmicindent\hskip\algorithmicindent\relax\(\phantom{\text{\textproc{RBT2}}(} U[\range{i_1}{i_2}, k], U[\range{i_2}{i_3}, k], V[\range{j_1}{j_2}, k], V[\range{j_2}{j_3}, k])\)
			\EndFor
		\EndFor
	\EndProcedure

	\Statex

	\Procedure{RBT2}{\(A_{11}, A_{12}, A_{21}, A_{22},U_1, U_2, V_1, V_2\){}}
		\State \(m_{b1}, n_{b1} \gets \mathrm{dim}(A_{11});\  m_{b2}, n_{b2} \gets \mathrm{dim}(A_{22})\) \Comment{Assuming \(m_{b1}\geq m_{b2}\) and \(n_{b1}\geq n_{b2}\)}
		\For{\(j\) from \(0\) to \(n_{b2}\)}
			\For{\(i\) from \(0\) to \(m_{b2}\)}
				\State \(a_{11} \gets A_{11}[i,j];\  a_{21} \gets A_{21}[i,j];\  a_{12} \gets A_{12}[i,j];\  a_{22} \gets A_{22}[i,j]\)

				\State \(A_{11}[i,j] \gets 2^{-1}(U_1[i]a_{11}V_1[j]+U_1[i]a_{12}V_2[j]+U_2[i]a_{21}V_1[j]+U_2[i]a_{22}V_2[j])\)
				\State \(A_{12}[i,j] \gets 2^{-1}(U_1[i]a_{11}V_1[j]-U_1[i]a_{12}V_2[j]+U_2[i]a_{21}V_1[j]-U_2[i]a_{22}V_2[j])\)
				\State \(A_{21}[i,j] \gets 2^{-1}(U_1[i]a_{11}V_1[j]+U_1[i]a_{12}V_2[j]-U_2[i]a_{21}V_1[j]-U_2[i]a_{22}V_2[j])\)
				\State \(A_{22}[i,j] \gets 2^{-1}(U_1[i]a_{11}V_1[j]-U_1[i]a_{12}V_2[j]-U_2[i]a_{21}V_1[j]+U_2[i]a_{22}V_2[j])\)
			\EndFor

			\For{\(i\) from \(m_{b2}\) to \(m_{b1}\)}

				\State \(a_{11} \gets A_{11}[i,j];\ a_{12} \gets A_{12}[i,j]\)

				\State \(A_{11}[i,j] \gets 2^{-1/2}(U_1[i]a_{11}V_1[j]+U_1[i]a_{12}V_2[j])\)
				\State \(A_{12}[i,j] \gets 2^{-1/2}(U_1[i]a_{11}V_1[j]-U_1[i]a_{12}V_2[j])\)
			\EndFor
		\EndFor
		\For{\(j\) from \(m_{b2}\) to \(m_{b1}\)}
			\For{\(i\) from \(0\) to \(m_{b2}\)}
				\State \(a_{11} \gets A_{11}[i,j];\ a_{21} \gets A_{21}[i,j]\)

				\State \(A_{11}[i,j] \gets 2^{-1/2}(U_1[i]a_{11}V_1[j]+U_2[i]a_{21}V_1[j])\)
				\State \(A_{21}[i,j] \gets 2^{-1/2}(U_1[i]a_{11}V_1[j]-U_2[i]a_{21}v_1[j])\)
			\EndFor
			\For{\(i\) from \(m_{b2}\) to \(m_{b1}\)}
				\State \(A_{11}[i,j] \gets U_1[i]A_{11}[i,j]V_1[j]\)
			\EndFor
		\EndFor
	\EndProcedure
	\end{algorithmic}
\end{algorithm*}%
This algorithm assumes \(U\) and \(V\) contain only the random coefficients, as per the packed storage used in previous works~\cite{lindquistReplacingPivotingDistributed2020}.
Line~\ref{alg:implementation:reference-size} controls the butterfly structure.
For a non-tiled implementation, setting \(m\gets 2^d\ceil{2^{-d}n}\) gives the smallest \(m\) greater or equal to \(n\) that is a multiple of \(2^d\).
For a tiled implementation (e.g., a distributed matrix) with a tile size
of \(n_b\), setting \(m\gets 2^dn_b\ceil{2^{-d}n_b^{-1}n}\) ensures that
the sizes of the constituent butterfly matrices are aligned to tile
boundaries.
Once \(m\) is chosen, the size of each butterfly and the corresponding loops easily follow.
However, loop splitting is needed to handle the difference between rows with 1 nonzero and rows with 2 nonzeros.
Note that all of the loops, except the outermost one on line~\ref{fig:implementation:loop-1} can be completely parallelized.
While \cref{alg:implementation} applies the butterflies one layer at a time, for small \(d\) it is possible to reorganize the algorithm to combine all \(d\) layers into a single pass over the matrix.
This would reduce the data movement by a factor of \(d\), but would require separate implementations for different~\(d\).

\section{Stability of GENP after RBT-preprocessing}
Successfully solving a system of equations with \ac{genp} requires that the diagonal elements in the factors are non-zero.
This occurs when each leading, principal submatrix is nonsingular.
In such a case, the matrix is called strongly nonsingular.
Parker proved that transforming a matrix on both sides with Parker
\acp{rbt} of depth \(\log_2(n)\) will result in a strongly nonsingular
matrix in exact arithmetic with probability 1~\cite{parkerRandomButterflyTransformations1995}.
We redo his analysis and theorems for our generalized \ac{rbt}.
Unfortunately, our formulation in its full generality gives weaker
conclusions than Parker's formulation, so we focus on semi-Parker
\acp{rbt} instead.
Our \cref{lemma:butterfly-poly-deg-1,lemma:U-poly-deg-1-and-nonzero,thm:RBT-gives-strongly-nonsingular-matrices} match Parker's \crefTheoremNamePlural~2, 3, and~4~\cite{parkerRandomButterflyTransformations1995},
respectively.

For ease of comparison with Parker's theory, we generally follow his notation~\cite{parkerRandomButterflyTransformations1995}.
The set of strictly increasing sequences of length \(k\) drawn from \(\{1, 2, \dots, m\}\) is denoted \(\seq{m}{k}\).
The subset of \(\seq{m}{k}\) consisting of sequences whose elements are consecutive modulo \(m\) is denoted \(\conseq{m}{k}\).
Addition and modulus are both applied element-wise.
The length of a sequence, \(\alpha\), is denoted \(|\alpha|\).
The relation \(\seqeq{}{}\) considers the sequences in sorted order and tests equality element-wise.
(Note that this accounts for elements' multiplicity.)
Similarly, the relation \({\seqeqmod{\alpha}{\beta}{m}}\) applies the modulus to each element before comparing the sequences as per \(\seqeq{}{}\).

\begin{lemma}[Cf.\ Parker's \crefTheoremName~2 \cite{parkerRandomButterflyTransformations1995}]
	\label{lemma:butterfly-poly-deg-1}
	For any \(1 \leq k \leq n\), let \({\alpha, \gamma \in \seq{n}{k}}\) and let
	\[
		c(n, \alpha, \gamma) = \det\bigl( O_{\mu,\nu}^{\btsize{n}}[\alpha,\gamma] \bigr)
	\]
	be a constant with \(O_{\mu,\nu}^{\btsize{n}}\) defined by \cref{eq:O-matrix-definition}.
	Then, \(c(n, \alpha, \gamma) = 0\) if and only if \(\seqneqmod{\alpha}{\gamma}{\mu+\nu}\).  Otherwise, \(2^{-|\alpha|/2} \leq |c(n, \alpha, \gamma)| \leq 1\).
\end{lemma}
\begin{proof}
	First note that for any \(j\in\alpha\), if \(j\not\in\gamma\pmod{\mu+\nu}\), then \(O_{\mu,\nu}^{\btsize{n}}[j,\gamma]=0\) and, thus, \(c(n,\alpha,\gamma)=0\).
	There is a similar implication for any \(j\in\gamma\).
	So, assume that \(j\in\alpha\pmod{\mu+\nu}\) if and only if \(j\in\gamma\pmod{\mu+\nu}\).
	Furthermore, if \(j\in\alpha\pmod{\mu+\nu}\)
with a multiplicity of 1 but \(j\in\gamma\pmod{\mu+\nu}\) with a multiplicity of 2, then the corresponding columns in \(O_{\mu,\nu}^{\btsize{n}}[\alpha,\gamma]\) are scalar multiples of each other, implying \(\det(O_{\mu,\nu}^{\btsize{n}}[\alpha,\gamma] = 0)\).
	The determinant is similarly zero if the multiplicities are flipped.
	Hence, if \(c(n, \alpha, \gamma) \neq 0\), then \({\seqeqmod{\alpha}{\gamma}{\mu+\nu}}\).

	Finally, assume \(\seqneqmod{\alpha}{\gamma}{\mu+\nu}\).
	Thus, \(O_{\mu,\nu}^{\btsize{n}}[\alpha,\gamma]\) has either one or two nonzeros per row.
	Rows with one nonzero element simply scale the determinant by \(\pm 1\) or \(\pm 2^{-1/2}\).
	Rows with two nonzero elements exist in pairs with nonzeros present in the same columns and no other nonzeros in those columns.
	So, because that \(2\times2\) submatrix is orthogonal, it scales the determinant by \(\pm 1\).
	Hence, the congruence assumption implies that \({2^{-|\alpha|/2} \leq |c(n, \alpha, \gamma)| \leq 1}\).
\end{proof}

\begin{lemma}[Cf.\ Parker's \crefTheoremName~3~\cite{parkerRandomButterflyTransformations1995}]
	\label{lemma:U-poly-deg-1-and-nonzero}
	Let \(\rbt{n}{U}\) be a semi-Parker \ac{rbt} of size \(n\) with \(d = \ceil{\log_2(n)}+1\). 
	Then for \(1 \leq k \leq n\) and all sequences \(\alpha, \beta \in \seq{n}{k}\), the determinant \(\det(\rbt{n}{U}[\alpha,\beta])\) is a polynomial of degree at most one
	in the random variables of the \ac{rbt} with coefficients based on \(\alpha\), \(\beta\), and \(k\).
	This polynomial is either zero or a function of a different set of random variables than \(\det(\rbt{n}{U}[\gamma, \delta])\) for any other \(\gamma, \delta \in \seq{n}{k}\).
	Furthermore, if \(\alpha=[1,2,\dots,k]\), then \(\det(\rbt{n}{U}[\alpha,\beta]) \neq 0\).
\end{lemma}
\begin{proof}

We prove the theorem by strong induction on \(n\).
The base case \(n=1\) is easy since \(\rbt{1}{U}\) is a \(1\times 1\) matrix with a single random value satisfying all of the stated properties.

The induction step can be proven by leveraging
the Binet-Cauchy theorem \cite[pg. 14]{marcusSurveyMatrixTheory1992}.  
Let \(\oplus\) denote the direct sum of two matrices, i.e.,
\[
	\rbt{\ell_0}{U}_0\!\oplus \rbt{\ell_1}{U}_1
	= \begin{bsmallmatrix} \rbt{\ell_0}{U}_0 & 0 \\ 0 & \rbt{\ell_1}{U}_1\end{bsmallmatrix}.
\]
Then, \(
	\rbt{n}{U} = \bigl(\rbt{\ell_0}{U}_0\!\oplus \rbt{\ell_1}{U}_1\bigr)O^{\btsize{n}}R
\)
with \(\ell_0+\ell_1 = n\).
Therefore,
\begin{align*}
	\det(\rbt{n}{U}[\beta,\alpha]) =&
	\sum_{\gamma\in \seq{n}{k}} \det\bigl( (\rbt{\ell_0}{U}_0\!\oplus \rbt{\ell_1}{U}_1)[\beta,\gamma] \bigr)
	\det\bigl(O^{\btsize{n}}[\gamma,\alpha]\bigr) \det\bigl(R[\alpha,\alpha]\bigr) \bigr)
\end{align*}
Furthermore, the determinant of \((\rbt{\ell_0}{U}_0\!\oplus \rbt{\ell_1}{U}_1)[\beta,\gamma]\) is zero unless there are subsequences \(\beta_0,\beta_1\) of \(1,\dots,\ell_0\) and \((\ell_0+1,\dots, n)\) whose sequence concatenation is \(\beta\) and likewise for \(\gamma_0\), \(\gamma_1\), and \(\gamma\) such that \(|\beta_0|=|\gamma_0|\) and \(|\beta_1|=|\gamma_1|\).
Then,
\begin{align*}
	&\det\bigl( (\rbt{\ell_0}{U}_0\!\oplus \rbt{\ell_1}{U}_1)[\beta,\gamma] \bigr)
	=
	\det\bigl( \rbt{\ell_0}{U}_0[\beta_0,\gamma_0] \bigr)
	\det\bigl( \rbt{\ell_1}{U}_1[\beta_1-\ell_0,\gamma_1-\ell_0] \bigr).
\end{align*}

First, by induction, these two smaller determinants are each a polynomial of degree at most one in their new (disjoint) sets of random variables.
By \cref{lemma:butterfly-poly-deg-1}, \(\det(O^{\btsize{n}}[\gamma,\alpha])\det(R[\alpha,\alpha])\) is also a polynomial of degree at most one in its variables, which are disjoint from those in \(\rbt{\ell_0}{U}_0\) and \(\rbt{\ell_1}{U}_1\).

Second, we can prove by induction that \acp{rbt} can be factored \(\rbt{n}{U} = R^{\btsize{1}}PR^{\btsize{n}}\),
where \(R^{\btsize{1}} = \mathrm{diag}(r_j^{\btsize{1}})\) and \(R^{\btsize{n}} = \mathrm{diag}(r_j^{\btsize{n}})\), and \(P\) is a product of other matrices.
So,
\[
	\det\bigl(\rbt{n}{U}[\beta,\alpha]\bigr) = \Bigl(\prod_{j\in\beta}r_j^{\btsize{1}}\Bigr) p(\beta,\alpha) \Bigl(\prod_{j\in\alpha}r_j^{\btsize{n}}\Bigr)
\]
where \(p(\beta,\alpha) = \det(P[\beta,\alpha])\) is a polynomial not involving the random variables \(r_j^{\btsize{1}}\) or \(r_j^{\btsize{n}}\).
So, \(\det( \rbt{n}{U}[\beta,\alpha])\) is nonzero if and only if \(p(\beta,\alpha)\) is nonzero.
And if \(\det( \rbt{n}{U}[\beta,\alpha])\) is nonzero, then substituting any other sequences \(\lambda,\delta\) for \(\beta,\alpha\) will yield a function of a different set of random variables.

Third, if \(\alpha = [1,2,\dots, k]\), we claim \(\det(\rbt{n}{U}[\beta,\alpha]) \neq 0\) for any fixed \(\beta\).
Continuing the factorization,
\[
	(\rbt{\ell_0}{U}_0\!\oplus \rbt{\ell_1}{U}_1)
	= R^{\btsize{1}}QR^{\btsize{\ell_0+\ell_1}}
\]
where \(QR^{\btsize{\ell_0+\ell_1}}O^{\btsize{n}} = P\).
So, by \cref{lemma:butterfly-poly-deg-1},
\begin{align*}
	\det\bigl(\rbt{n}{U}[\beta,\alpha]\bigr)
	=&
	\sum_{\gamma\in\seq{n}{k}}
	\Bigl(\prod_{j\in\beta}r_j^{\btsize{1}}\Bigr)
	q(\beta,\gamma)
	\Bigl(\prod_{j\in\gamma}r_j^{\btsize{\ell_0+\ell_1}}\Bigr)
	c(n,\gamma,\alpha)
	\Bigl(\prod_{j\in\alpha}r_j^{\btsize{n}}\Bigr)
\end{align*}
with \(q(\beta,\gamma) = \det(Q[\beta,\gamma])\).
The contribution to this sum by any given \(\gamma\) is thus either zero or a function of a different set of random variables than any other \(\gamma\).
Then, all that is needed to complete the proof is to find a value of \(\gamma\) for which both \(q[\beta,\gamma]\) and \(c(n,\gamma,\alpha)\) are nonzero since no other value can affect its contribution to the sum.

If \(\beta_0,\beta_1\) are subsequences of \(1,\dots,\ell_0\) and \(\ell_0+1,\dots,n\) whose sequence concatenation is \(\beta\),
then it is sufficient to show that there always exist \emph{consecutive} subsequences \(\gamma_0,\gamma_1\) of \(1,\dots,\ell_0\) and \(\ell_0+1,\dots,n\) such that \(\gamma\) is their concatenation, \(c(n, \gamma, \alpha)\neq 0\), and also \(|\beta_0| = |\gamma_0|\) and \(|\beta_1| = |\gamma_1|\).
(Although, \(\gamma\) itself may be nonconsecutive.)
One such pair of sequences is
\begin{align*}
	\gamma_1 &= [1, 2, \dots, |\beta_1|] + \mu+\nu \\
	\gamma_0 &= [|\beta_1|+1, \dots, |\beta_1|+|\beta_0|] \bmod{\mu+\nu}.
\end{align*}
Recall that the partitioning of \(\beta\) implies \(|\beta_1| < \mu\).
So, both \(\gamma_0\) and \(\gamma_1\) are trivially (round-robin)
consecutive.
Furthermore, \(\seqeqmod{\gamma}{\alpha}{\mu+\nu}\), and \cref{lemma:butterfly-poly-deg-1} implies that \(c(n,\gamma,\alpha)\neq 0\).
By assumption, the butterflies in \(\rbt{\ell_0}{U}_0[\beta_0,\gamma_0]\) are all Parker butterflies and, thus, its size is a power of 2.
Hence, by Parker's \crefTheoremName~3 and induction respectively, the determinants of \(\rbt{\ell_0}{U}_0[\beta_0,\gamma_0]\) and \(\rbt{\ell_1}{U}_1[\beta_1,\gamma_1]\) are nonzero polynomials.
\end{proof}

\begin{theorem}[Cf.\ Parker's \crefTheoremName~4 \cite{parkerRandomButterflyTransformations1995}]
	\label{thm:RBT-gives-strongly-nonsingular-matrices}
	Let \(A\) be a nonsingular matrix, and let \(\rbt{n}{U}\) and \(\rbt{n}{V}\) be conformant, depth-\(\ceil{\log_2(n)}+1\), semi-Parker \acp{rbt}.
	Then, with probability~1, \(\rbt[T]{n}{U}\!\!A\rbt{n}{V}\) is strongly nonsingular.
\end{theorem}
\begin{proof}
	This theorem is a generalization of Parker's \crefTheoremName~4~\cite{parkerRandomButterflyTransformations1995}, and the proof follows identically, except for the use of \cref{lemma:U-poly-deg-1-and-nonzero} instead of Parker's equivalent.
\end{proof}

Unfortunately, this theorem does not guarantee high accuracy in a finite-precision setting.
The primary source of instability is element growth that leads to cancellation errors.
Using the Binet-Cauchy theorem, Parker suggested that the growth factor of the randomized matrix is, in some sense, the average growth over ``all possible pivoting sequences (good and bad) with the original matrix''~\cite[p.~14]{parkerRandomButterflyTransformations1995}, which can be extended to our generalized formulation.
Alternatively, using Schur's identity and the interlace theorem for singular-values, we can show that
\begin{equation}
	\label{eq:growth-bound}
	|\widehat{A}^{(k)}[i,j]| \leq \frac{\sqrt{n}\norm{2}{\widehat{A}}}{\sigma_k(\widehat{A}[1:k, 1:k])} \max_{ij} |\widehat{A}[i, j]|.
\end{equation}
Thus, the growth is determined by \(\sigma_k(\widehat{A}[1:k, 1:k])\).
For preprocessing with Gaussian transforms, Pan and Zhao have bounded this value to \(\Oh{n^{-3/2}}\) with high probability~\cite[Thm.~4.3]{panNumericallySafeGaussian2017};
unfortunately, bounding it for \ac{rbt} preprocessing is still an open problem.

For depths less than \(\ceil{\log_2(n)}+1\), not even theoretical strong nonsingularity is guaranteed.
For example, consider
\begin{equation}
	\label{eq:example-for-small-d}
	A = \begin{bmatrix}
		0 & 1 & 0 & 0 \\
		1 & 0 & 0 & 0 \\
		0 & 0 & 0 & 1 \\
		0 & 0 & 1 & 0 \\
	\end{bmatrix}\!,\
	U = \begin{bmatrix}
		r_1 & 0 & r_3 & 0 \\
		0 & r_2 & 0 & r_4 \\
		r_1 & 0 & -r_3 & 0 \\
		0 & r_2 & 0 & -r_4 \\
	\end{bmatrix}\!, \text{ and }
	V = \begin{bmatrix}
		r_5 & 0 & r_7 & 0 \\
		0 & r_6 & 0 & r_8 \\
		r_5 & 0 & -r_7 & 0 \\
		0 & r_6 & 0 & -r_8 \\
	\end{bmatrix}\!.
\end{equation}
Even though \(A\) is nonsingular (and even orthogonal), applying these depth-1 \acp{rbt} gives
\[
	U^T\!AV = \begin{bmatrix}
		0 & 2r_1r_6 & 0 & 0 \\
		2r_2r_5 & 0 & 0 & 0 \\
		0 & 0 & 0 & 2r_3r_8 \\
		0 & 0 & 2r_4r_7 & 0
	\end{bmatrix}\!,
\]
which is obviously not strongly nonsingular.
However, previous experimental results have shown that a depth-2 transform can usually achieve a reasonable accuracy in practice~\cite{beckerReducingAmountPivoting2012,baboulinAcceleratingLinearSystem2013}.

\section{Recovering Accuracy in the Solution Vector}
\label{sec:recovering-acc}

As discussed in the previous section, this approach does not guarantee high accuracy, even probabilistically.
Thus, a production-quality solver must be able to recover accuracy when the growth is too large.
Iterative refinement is generally used for this purpose but relying
upon it exclusively is often
unsuccessful~\cite{liMakingSparseGaussian1998}, especially the stationary
formulation that can only recover minor errors. Such a refinement
scheme can fail to converge if the inner solver is too inaccurate
relative to the matrix's condition
number~\cite{carsonAcceleratingSolutionLinear2018}.
Restarted GMRES can provide a more robust convergence at the cost of extra computation per iteration~\cite{carsonNewAnalysisIterative2017}.

When refinement cannot recover enough accuracy, the problem must be
re-solved with a more robust solver, such as \ac{gepp}, QR
factorization, or even a mixture of both~\cite{favergeMixingLUQR2015}.
This makes the solver much slower for problematic matrices;
however, if the \ac{rbt} solver is usually accurate and significantly faster in successful cases (as is shown in \cref{sec:exp}),
the average time to solution will be lower than that of \ac{gepp}.
Furthermore, in successful cases, this mechanism has no overhead
since iterative refinement already needs a copy of the original matrix and checks the solution's accuracy.
This fallback mechanism is an example of the ``responsibly reckless''
paradigm where a fast, reckless algorithm is combined with a responsible verification of the result~\cite{dongarraExtremeComputingRules2017}.

Recent work on mixed-precision techniques advances the
numerical analysis of the iterative refinement while fully utilizing the
underlying hardware, using up to 5 floating-point precisions
simultaneously to recover the lost digits in a method called
GMRES-IR5~\cite{amestoy2021gmresir5}. The use of these methods would
further improve the accuracy of the solution in addition to our proposed
Generalized \ac{rbt}. To make the numerical properties of our algorithm
clear, we focus our evaluation on only the simplest form of
accuracy-recovering techniques.

\section{Experimental Results}
\label{sec:exp}

To understand this generalization experimentally, we implemented an \ac{rbt}-based solver in the \ac{slate} library~\cite{gatesSLATEDesignModern2019} and tested it on the Summit supercomputer at Oak Ridge National Laboratory.
\Ac{slate} is a dense linear algebra library designed to replace ScaLAPACK in the increasingly heterogeneous landscape of high-performance computing.
\Ac{slate} stores matrices in tiles which are distributed across the processes, usually in a 2D block-cyclic fashion.
The initial implementation of the proposed generalized \ac{rbt} is
available in the \texttt{2023.11.05} release of SLATE, including
the fallback strategy discussed in \cref{sec:recovering-acc}.

Our \ac{rbt} implementation is based on \eqref{eq:rbt-depth-d} and aligns the butterfly structure to these tiles in order to avoid the expensive tile management of other distributed implementations~\cite{baboulinEfficientDistributedRandomized2014,lindquistReplacingPivotingDistributed2020}.
In other words, the top half of the first layer (i.e., the layer with a
single butterfly matrix) is chosen to have a dimension equal to the
first \(2^{d-1}\ceil{n_t2^{-d}}\) tiles where \(n_t = \ceil{n/n_b}\)
is the total number of rows of tiles in \(A\).
For simplicity, our implementation assumes that all tiles, except the last row and column, are the same size.
For non-uniform tile sizes, it deviates from \cref{eq:O-matrix-definition} to keep the butterflies aligned to tiles;
however, we did not consider the variable-tile-size case any further since we are
unaware of any applications that effectively use non-uniform tiles
in a distributed, dense factorization.
As in previous works, the random variables are taken as \(\exp(r/20)\) where \(r\) is uniformly selected from \([-1, 1]\)~\cite{baboulinAcceleratingLinearSystem2013}.
While our overall solver is GPU-enabled, our \ac{rbt} implementation only uses the CPU.
This choice is based on the fact that the \ac{rbt} kernel is inherently
memory-bound and the MPI implementations available on Summit have to
transfer data through the host.\footnote{
While the MPI implementations are capable of handling native GPU
pointers, hardware limitations prevent GPU-Direct transfers between
the GPU memories and the network without passing the data through host CPU memory.}
Thus, any speedup gained by doing the computation on the GPU devices would
be completely overshadowed by the overhead of transferring the data
between the CPU memory and the attached GPU devices.

We test the generalized \ac{rbt} solver against our Gaussian
elimination with no pivoting (\acs{genp}), tournament pivoting
(\acs{getp}), and partial pivoting (\acs{gepp}) implementations.
The first provides the best performance but is numerically unstable,
while the other two provide better numerical stability but worse performance.
Note that the performance of \ac{gepp} also provides bounds on the performance of threshold pivoting.
Additionally, we compare it with Parker \acp{rbt}, based on our previous implementation~\cite{lindquistReplacingPivotingDistributed2020}.
Note that the fallback solver was always disabled to focus on the
effects of the \ac{rbt} alone.

\subsection{Experimental Configuration}
Summit is based on IBM's POWER System AC922 node.
Each node contains two 22-core, IBM POWER9 CPUs and six NVIDIA Volta V100 GPUs which are evenly divided between two sockets.
Most of the computational capacity comes from the GPUs, each providing
peak
\SI{7.45}{\tera\flop\per\second} (80 SMs of 64 CUDA cores clocked at
1.455 GHz) and
\SI{16}{\gibi\byte} \ac{hbm2} providing
\SI{900}{\giga\byte\per\second} peak main-memory bandwidth.
Each CPU provides \SI{540}{\giga\flop\per\second} peak performance,
\SI{256}{\gibi\byte} DDR4 memory,
and \SI{170}{\giga\byte\per\second} peak main-memory bandwidth.
Note that 1 core of each CPU is reserved for overheads associated with
OS tasks and is not accessible by the user applications.
NVLINK provides a peak bidirectional \SI{50}{\giga\byte\per\second}
transfer rate between the CPU and GPUs in a socket.
A dual-rail EDR InfiniBand network connects the nodes with a bandwidth of \SI{23}{\giga\byte\per\second}.
All nodes use Red Hat Enterprise Linux (RHEL) version 8.2.

The code was compiled with GCC 9.1.0 and CUDA 11.0.3.
It was linked against IBM's Spectrum MPI 10.4.0.3, ESSL 6.1.0-2,
Netlib LAPACK 3.8.0, and Netlib ScaLAPACK 2.1.0.
The code used for these experiments and the output files are available online at \url{https://figshare.com/s/4c4b0ffd1ef44484e884}.

Each job is submitted to the scheduler with the flags \texttt{-nnodes~8 -alloc\_flags~smt1}; the second flag disables simultaneous multithreading.
Processes were launched with ORNL-provided utility
\texttt{jsrun -n~16 -a~1 -c~21 -g~3 -b~packed:21 -d~packed}, which allocates one process per socket and binds the threads to the corresponding CPUs and limits CUDA to the corresponding GPUs.
The tester code from \ac{slate} was used for all performance and accuracy results, with a modification to print the accuracy as per \cref{eq:inf-norm-backerr}.
We tuned \ac{slate}'s parameters using 8 nodes and \(n=\num{150000}\).
All tests were configured with \texttt{-{}-origin~h -{}-target~d -{}-type~d -{}-ref~n -{}-lookahead~2 -{}-ib~64 -{}-nrhs~1}. In other words:
\begin{itemize}
\item The matrix and vector data originates on the host.
\item The majority of computation is done on the GPUs.
\item Double precision is used for matrix and vector elements.
\item The ScaLAPACK reference implementation is not run.
\item Two lookahead tasks are used with depth
2~\cite[Sec.~5]{gatesSLATEDesignModern2019}.
\item A blocking factor of 64 is used when factoring the panel or the diagonal tile.
\item There was one right-hand side vector \(b\) and a corresponding
single vector of unknowns \(x\).
\end{itemize}
The matrix generator is configured with \texttt{-{}-matrixB~rand
-{}-seed~42 -{}-seedB~64} plus \texttt{-{}-matrix} option with the
appropriate argument shown in the first column of
Tab.~\ref{tab:acc:back-error};
this results in the same random right-hand side for each test with the
same problem size and ensures that the system matrices are reproducible.
\Ac{genp} and the \ac{rbt}-solvers were configured with \texttt{-{}-nb~512 -{}-fallback~n}, i.e., a tile size of 512 and the fallback mechanism is disabled.
Note that our tests always use the \texttt{gesv\_rbt} routine for \ac{genp} results; overheads for the \ac{rbt} and the iterative refinement are only present when the depth or iteration limit, respectively, are larger than 0.
\Ac{gepp} was configured with \texttt{-{}-nb~896 -{}-panel-threads~20}, i.e., a tile size of 896 and 20 threads to factor the panels.
\Ac{getp} was configured with \texttt{-{}-nb~768 -{}-panel-threads~2}, i.e., a tile size of 768 and 2 threads to do the tournament reduction.
Finally, the flags \verb|--check| and \verb|--dim| were configured as appropriate for each test, as well as \verb|--refine| and \verb|--depth| for the \ac{genp} and \ac{rbt} solvers.
Because MPI and BLAS libraries often initialize their internal state and
allocate temporary buffers on the first call, warm-up tests of size
\num{10000} were run first during each test where performance was
measured. This avoided uniform overheads across or measurements without
the influence of library initialization.

\subsection{Accuracy}
\label{sec:exp:acc}

To demonstrate the accuracy of the approach, we tested the solvers with a variety of matrices and present the results in \cref{tab:acc:back-error}.
The accuracy was measured with the \(\infty\)-norm backward error:
\begin{equation}
	\label{eq:inf-norm-backerr}
	\frac{\norm{\infty}{b - Ax}}{\norm{\infty}{A}\norm{\infty}{x} + \norm{\infty}{b}}.
\end{equation}
Problems were of size \(\num{150000} = 292\times 512 + 496\), which
results in the generalized \ac{rbt} method blocking the matrix
differently than the original Parker \ac{rbt}.
\begin{table*}
	\centering
	\caption{Comparison between the \ac{rbt} solver, \ac{gepp},
        \ac{getp}, and \ac{genp} for the \(\infty\)-norm backward error
        for various matrices of size \(n=\num{150000}\).}
	\label{tab:acc:back-error}
	\setlength\tabcolsep{4pt}
	\begin{tabular}{r*{8}{S[table-format=1.1e+2,round-mode=places,table-auto-round,scientific-notation=true,retain-zero-exponent=true,tight-spacing]}}  \toprule  Matrix & {\ac{gepp}} & {\ac{getp}} & {\ac{genp}} & {\ac{rbt}}  & {\ac{rbt}} & {Parker \ac{rbt}}  & {Parker \ac{rbt}} \\         &             &             &             & {refined}   &            & {refined}          &                   \\  \midrule
\matname{rand+nI} & 1.95e-14 & 1.87e-14 & 1.84e-14 & 2.06e-16 & 2e-14 & 2.06e-16 & 1.81e-14 \\
\matname{rand} & 2.52e-14 & 4.18e-14 & 2.36e-10 & 2.7e-17 & 2.09e-10 & 2.65e-17 & 6.98e-11 \\
\matname{rands} & 5.61e-14 & 8.06e-14 & 2.71e-10 & 4.36e-17 & 1.25e-09 & 4.4e-17 & 1.55e-09 \\
\matname{randn} & 5.34e-14 & 9.92e-14 & 6.84e-10 & 4.39e-17 & 3.43e-10 & 3.86e-17 & 3.45e-10 \\
\matname{randb} & 4.01e-14 & 6.21e-14 & {NaN} & 2.17e-17 & 7.46e-10 & 2.53e-17 & 1.58e-09 \\
\matname{randr} & 4.44e-14 & 7.53e-14 & {NaN} & 3.47e-17 & 1.62e-09 & 1.87e-15 & 4.9e-08 \\
\matname{chebspec} & 3.4e-16 & 2.95e-16 & 2.04e-09 & 1.06e-17 & 9.83e-14 & 2.69e-17 & 2.75e-13 \\
\matname{circul} & 1.43e-17 & 1.62e-17 & 9.45e-14 & 2.52e-18 & 1e-17 & 2.54e-18 & 9.03e-18 \\
\matname{fiedler} & 2.13e-18 & 1.03e-17 & {NaN} & 2.26e-18 & 1.8e-17 & 2e-18 & 7.66e-17 \\
\matname{gfpp} & {NaN} & {NaN} & {NaN} & 7.98e-19 & 3.74e-17 & 1.14e-18 & 1.78e-17 \\
\matname{orthog} & 3.24e-17 & 3.16e-17 & 2.8e-05 & 0.00101 & 0.000838 & 0.000841 & 0.000918 \\
\matname{ris} & 1.41e-16 & 9.82e-17 & 0.105 & 0.146 & 0.159 & 0.135 & 0.135 \\
\matname{riemann} & 4.08e-14 & 5.66e-14 & 2.78e-12 & 0.00143 & 0.00146 & 5.79e-08 & 0.00118 \\
\matname{zielkeNS} & 3.47e-19 & 1.05e-18 & {NaN} & 8.91e-19 & 1.29e-17 & 1.06e-18 & 1.98e-16 \\
\bottomrule\end{tabular}

\end{table*}%
Both formulations of the \ac{rbt} solver were tested with and without two steps of iterative refinement.
Additionally, allowing iterative refinement to follow \ac{genp} enabled
it to achieve a double-precision solution in all cases except
\matname{orthog}, \matname{ris}, and those with NaN values.\footnote{The
backward errors for \ac{genp} with iterative refinement can be found in
our results' artifact.}
With iterative refinement, the butterfly solver provided similar or better accuracy than \ac{gepp} for eleven of the fourteen matrices.
It was even able to solve the \matname{gfpp} matrix for which \ac{gepp} has catastrophic element growth that overflows double-precision.
The three matrices with worse accuracy are analyzed in \cref{sec:exp:orr}.
Furthermore, any \ac{rbt} provides significantly better accuracy than
\ac{genp} for six matrices and provides significantly worse accuracy
for only one matrix (\matname{riemann}).
Of the problems where two steps of iterative refinement did not reach
double-precision accuracy, the refinement provided negligible benefit
compared to the plain factorization in all cases except the one with
the Parker \ac{rbt}.

Additionally, the generalized formulation of the \ac{rbt} gave similar accuracy to Parker's formulation.
The one exception is that, for \matname{riemann}, two steps of iterative refinement were able to improve the accuracy halfway to double-precision for Parker's formulation but not ours.
(Continuing the iterative refinement allowed the Parker \ac{rbt}-solver to reach double-precision accuracy after five iterations.)
This difference is surprising, given the similarity in initial backward error;
we suspect there is a subtle interaction between the \ac{rbt} structure, its resulting errors, and the singular- or eigen-vectors of the matrix.

To further demonstrate the validity of our generalized \ac{rbt} structure, \cref{fig:exp:acc-gen-bt-struct-chebspec,fig:exp:acc-gen-bt-struct-circul,fig:exp:acc-gen-bt-struct-fiedler} show the accuracy of the \ac{rbt} solver for three matrices of varying matrix sizes.
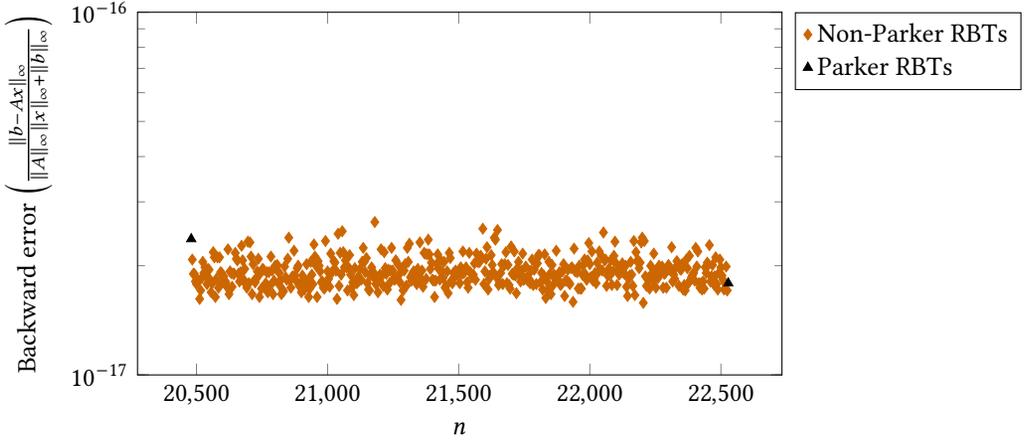
\begin{figure}
	\centering
	\begin{tikzpicture}\begin{semilogyaxis}[width=4in, height=2.5in,ylabel=Backward error \(\left(\frac{\norm{\infty}{b-Ax}}{\norm{\infty}{A}\norm{\infty}{x} + \norm{\infty}{b}}\right)\),ymin=1e-17, ymax=1e-16,xlabel=\(n\),xticklabel style={/pgfplots/scaled x ticks=false},legend style={at={(1.02,1)},anchor=north west},legend cell align=left,]
\addplot[only marks,orange!80!black,mark=diamond*] coordinates {(20484,2.08e-17) (20488,1.9e-17) (20492,1.9e-17) (20496,1.81e-17) (20500,1.78e-17) (20504,1.82e-17) (20508,1.82e-17) (20512,1.62e-17) (20516,1.71e-17) (20520,1.85e-17) (20524,2.05e-17) (20528,1.91e-17) (20532,1.98e-17) (20536,1.87e-17) (20540,1.78e-17) (20544,1.96e-17) (20548,1.95e-17) (20552,1.83e-17) (20556,1.69e-17) (20560,1.79e-17) (20564,2.18e-17) (20568,2.11e-17) (20572,1.64e-17) (20576,1.84e-17) (20580,2.11e-17) (20584,1.74e-17) (20588,2.11e-17) (20592,1.92e-17) (20596,1.91e-17) (20600,1.8e-17) (20604,1.87e-17) (20608,1.82e-17) (20612,1.8e-17) (20616,1.75e-17) (20620,2.22e-17) (20624,1.67e-17) (20628,1.93e-17) (20632,1.91e-17) (20636,1.75e-17) (20640,1.9e-17) (20644,1.9e-17) (20648,2.17e-17) (20652,2.09e-17) (20656,2.07e-17) (20660,1.9e-17) (20664,1.81e-17) (20668,1.92e-17) (20672,2.27e-17) (20676,2.07e-17) (20680,1.83e-17) (20684,1.96e-17) (20688,1.78e-17) (20692,1.74e-17) (20696,2.33e-17) (20700,1.79e-17) (20704,2.32e-17) (20708,2.12e-17) (20712,1.67e-17) (20716,2.02e-17) (20720,1.82e-17) (20724,1.8e-17) (20728,1.84e-17) (20732,1.93e-17) (20736,2.09e-17) (20740,1.64e-17) (20744,1.91e-17) (20748,1.73e-17) (20752,1.96e-17) (20756,1.95e-17) (20760,1.75e-17) (20764,1.79e-17) (20768,1.92e-17) (20772,1.69e-17) (20776,1.92e-17) (20780,1.76e-17) (20784,2.12e-17) (20788,1.94e-17) (20792,1.67e-17) (20796,1.82e-17) (20800,2e-17) (20804,1.96e-17) (20808,2.01e-17) (20812,1.79e-17) (20816,1.91e-17) (20820,1.82e-17) (20824,1.67e-17) (20828,1.72e-17) (20832,1.7e-17) (20836,1.99e-17) (20840,1.86e-17) (20844,1.83e-17) (20848,2.09e-17) (20852,2.39e-17) (20856,1.82e-17) (20860,1.74e-17) (20864,1.86e-17) (20868,1.87e-17) (20872,2.2e-17) (20876,1.88e-17) (20880,1.82e-17) (20884,1.97e-17) (20888,2e-17) (20892,1.76e-17) (20896,1.76e-17) (20900,1.69e-17) (20904,1.88e-17) (20908,1.83e-17) (20912,1.72e-17) (20916,1.81e-17) (20920,2.02e-17) (20924,1.82e-17) (20928,1.64e-17) (20932,1.73e-17) (20936,2.05e-17) (20940,1.89e-17) (20944,1.78e-17) (20948,2.29e-17) (20952,2.12e-17) (20956,2e-17) (20960,1.64e-17) (20964,2.06e-17) (20968,2e-17) (20972,1.93e-17) (20976,1.8e-17) (20980,1.81e-17) (20984,1.74e-17) (20988,1.73e-17) (20992,2.34e-17) (20996,1.81e-17) (21000,1.93e-17) (21004,1.84e-17) (21008,1.75e-17) (21012,2.16e-17) (21016,1.92e-17) (21020,1.93e-17) (21024,2.06e-17) (21028,2.04e-17) (21032,1.94e-17) (21036,1.62e-17) (21040,2.45e-17) (21044,1.76e-17) (21048,2.17e-17) (21052,1.89e-17) (21056,2.49e-17) (21060,2.18e-17) (21064,1.98e-17) (21068,2.07e-17) (21072,1.98e-17) (21076,1.71e-17) (21080,1.98e-17) (21084,2.14e-17) (21088,1.89e-17) (21092,1.74e-17) (21096,1.8e-17) (21100,1.66e-17) (21104,1.97e-17) (21108,1.81e-17) (21112,1.79e-17) (21116,1.86e-17) (21120,1.82e-17) (21124,1.96e-17) (21128,1.83e-17) (21132,2.05e-17) (21136,2.24e-17) (21140,2.05e-17) (21144,1.92e-17) (21148,2.1e-17) (21152,1.79e-17) (21156,1.72e-17) (21160,1.71e-17) (21164,1.88e-17) (21168,1.98e-17) (21172,1.93e-17) (21176,1.86e-17) (21180,2.64e-17) (21184,1.78e-17) (21188,1.86e-17) (21192,1.69e-17) (21196,1.83e-17) (21200,1.89e-17) (21204,2.22e-17) (21208,1.88e-17) (21212,1.76e-17) (21216,2.18e-17) (21220,1.97e-17) (21224,1.9e-17) (21228,1.72e-17) (21232,1.95e-17) (21236,1.79e-17) (21240,1.94e-17) (21244,2.24e-17) (21248,1.94e-17) (21252,1.84e-17) (21256,1.83e-17) (21260,2.07e-17) (21264,1.85e-17) (21268,2.02e-17) (21272,1.89e-17) (21276,1.8e-17) (21280,1.61e-17) (21284,1.96e-17) (21288,1.71e-17) (21292,1.94e-17) (21296,2.02e-17) (21300,1.76e-17) (21304,2.24e-17) (21308,2.05e-17) (21312,2.3e-17) (21316,1.86e-17) (21320,1.82e-17) (21324,1.9e-17) (21328,1.93e-17) (21332,1.92e-17) (21336,1.87e-17) (21340,1.78e-17) (21344,1.96e-17) (21348,1.84e-17) (21352,2.32e-17) (21356,1.84e-17) (21360,1.94e-17) (21364,1.75e-17) (21368,2.17e-17) (21372,2.16e-17) (21376,1.84e-17) (21380,2.15e-17) (21384,2.36e-17) (21388,2.15e-17) (21392,1.96e-17) (21396,2.04e-17) (21400,1.83e-17) (21404,1.82e-17) (21408,1.64e-17) (21412,1.82e-17) (21416,2.02e-17) (21420,2.15e-17) (21424,1.96e-17) (21428,1.83e-17) (21432,2.07e-17) (21436,1.93e-17) (21440,1.85e-17) (21444,2.05e-17) (21448,1.74e-17) (21452,1.76e-17) (21456,1.93e-17) (21460,1.68e-17) (21464,2.17e-17) (21468,2.1e-17) (21472,1.82e-17) (21476,1.79e-17) (21480,1.91e-17) (21484,1.86e-17) (21488,1.94e-17) (21492,1.89e-17) (21496,1.9e-17) (21500,1.93e-17) (21504,1.79e-17) (21508,1.98e-17) (21512,1.9e-17) (21516,2.19e-17) (21520,1.91e-17) (21524,1.94e-17) (21528,1.98e-17) (21532,1.8e-17) (21536,1.74e-17) (21540,2.17e-17) (21544,1.92e-17) (21548,1.9e-17) (21552,1.86e-17) (21556,1.89e-17) (21560,2.03e-17) (21564,1.83e-17) (21568,2.16e-17) (21572,2.05e-17) (21576,2.05e-17) (21580,2.05e-17) (21584,2.13e-17) (21588,2.1e-17) (21592,2.53e-17) (21596,1.78e-17) (21600,1.93e-17) (21604,2.09e-17) (21608,2.25e-17) (21612,1.72e-17) (21616,1.84e-17) (21620,1.88e-17) (21624,2.14e-17) (21628,2.01e-17) (21632,1.92e-17) (21636,2.39e-17) (21640,2.36e-17) (21644,1.86e-17) (21648,2.51e-17) (21652,1.77e-17) (21656,1.8e-17) (21660,1.91e-17) (21664,2.02e-17) (21668,1.78e-17) (21672,2e-17) (21676,2.02e-17) (21680,2.18e-17) (21684,1.79e-17) (21688,2.01e-17) (21692,1.77e-17) (21696,1.73e-17) (21700,2.26e-17) (21704,1.9e-17) (21708,1.89e-17) (21712,2.19e-17) (21716,1.88e-17) (21720,1.68e-17) (21724,1.97e-17) (21728,2.05e-17) (21732,1.89e-17) (21736,1.81e-17) (21740,1.94e-17) (21744,2.16e-17) (21748,1.83e-17) (21752,1.84e-17) (21756,1.99e-17) (21760,1.86e-17) (21764,1.82e-17) (21768,1.84e-17) (21772,1.79e-17) (21776,1.94e-17) (21780,2.08e-17) (21784,1.77e-17) (21788,1.95e-17) (21792,2.05e-17) (21796,1.78e-17) (21800,2.03e-17) (21804,2.14e-17) (21808,2.14e-17) (21812,1.8e-17) (21816,1.91e-17) (21820,1.72e-17) (21824,1.88e-17) (21828,2.15e-17) (21832,1.63e-17) (21836,1.81e-17) (21840,1.72e-17) (21844,2.05e-17) (21848,1.9e-17) (21852,1.67e-17) (21856,2.01e-17) (21860,1.77e-17) (21864,1.89e-17) (21868,1.84e-17) (21872,2.2e-17) (21876,1.83e-17) (21880,1.88e-17) (21884,1.92e-17) (21888,1.77e-17) (21892,1.78e-17) (21896,1.93e-17) (21900,2.02e-17) (21904,1.65e-17) (21908,1.82e-17) (21912,1.82e-17) (21916,1.95e-17) (21920,1.8e-17) (21924,1.78e-17) (21928,1.97e-17) (21932,1.98e-17) (21936,1.59e-17) (21940,2.23e-17) (21944,1.93e-17) (21948,1.95e-17) (21952,1.92e-17) (21956,1.73e-17) (21960,2e-17) (21964,1.73e-17) (21968,1.73e-17) (21972,1.89e-17) (21976,2.01e-17) (21980,2.08e-17) (21984,1.89e-17) (21988,1.83e-17) (21992,2.07e-17) (21996,2.31e-17) (22000,2.02e-17) (22004,1.86e-17) (22008,2.19e-17) (22012,1.82e-17) (22016,1.84e-17) (22020,1.91e-17) (22024,1.79e-17) (22028,2.04e-17) (22032,1.92e-17) (22036,2.19e-17) (22040,1.95e-17) (22044,2.13e-17) (22048,1.94e-17) (22052,2.47e-17) (22056,2.07e-17) (22060,2.08e-17) (22064,1.76e-17) (22068,1.88e-17) (22072,1.96e-17) (22076,1.88e-17) (22080,2.07e-17) (22084,1.77e-17) (22088,1.93e-17) (22092,2.13e-17) (22096,2.33e-17) (22100,2.03e-17) (22104,1.77e-17) (22108,1.93e-17) (22112,1.89e-17) (22116,1.95e-17) (22120,1.92e-17) (22124,2.03e-17) (22128,1.92e-17) (22132,1.9e-17) (22136,2e-17) (22140,1.75e-17) (22144,1.65e-17) (22148,1.84e-17) (22152,1.88e-17) (22156,2.05e-17) (22160,1.89e-17) (22164,1.76e-17) (22168,2.34e-17) (22172,1.66e-17) (22176,2.18e-17) (22180,2.06e-17) (22184,2.07e-17) (22188,1.97e-17) (22192,2.04e-17) (22196,2.32e-17) (22200,2.39e-17) (22204,1.58e-17) (22208,2.34e-17) (22212,1.87e-17) (22216,1.76e-17) (22220,1.84e-17) (22224,1.85e-17) (22228,1.73e-17) (22232,2e-17) (22236,1.9e-17) (22240,1.82e-17) (22244,2.16e-17) (22248,1.93e-17) (22252,1.74e-17) (22256,1.95e-17) (22260,1.85e-17) (22264,2e-17) (22268,1.94e-17) (22272,1.93e-17) (22276,1.99e-17) (22280,1.74e-17) (22284,1.9e-17) (22288,1.91e-17) (22292,1.74e-17) (22296,1.95e-17) (22300,2.01e-17) (22304,1.95e-17) (22308,1.95e-17) (22312,2.25e-17) (22316,1.81e-17) (22320,1.89e-17) (22324,1.84e-17) (22328,1.98e-17) (22332,2.06e-17) (22336,2.01e-17) (22340,1.76e-17) (22344,1.96e-17) (22348,1.95e-17) (22352,2.02e-17) (22356,1.72e-17) (22360,1.77e-17) (22364,1.71e-17) (22368,2.08e-17) (22372,1.97e-17) (22376,1.79e-17) (22380,2.13e-17) (22384,1.83e-17) (22388,1.86e-17) (22392,1.89e-17) (22396,1.91e-17) (22400,1.82e-17) (22404,2.24e-17) (22408,1.99e-17) (22412,2.19e-17) (22416,1.79e-17) (22420,1.74e-17) (22424,1.79e-17) (22428,1.72e-17) (22432,1.81e-17) (22436,1.73e-17) (22440,1.98e-17) (22444,1.91e-17) (22448,2.04e-17) (22452,2.01e-17) (22456,2.28e-17) (22460,1.81e-17) (22464,1.87e-17) (22468,2.01e-17) (22472,1.78e-17) (22476,1.87e-17) (22480,2.11e-17) (22484,1.78e-17) (22488,1.72e-17) (22492,1.91e-17) (22496,2.08e-17) (22500,1.98e-17) (22504,1.88e-17) (22508,1.83e-17) (22512,1.72e-17) (22516,1.79e-17) (22520,1.99e-17) (22524,1.71e-17)};
\addlegendentry{Non-Parker \acsp{rbt}};
\addplot[only marks,black,mark=triangle*] coordinates {(20480,2.37e-17) (22528,1.79e-17)};
\addlegendentry{Parker \acsp{rbt}};
\end{semilogyaxis} \end{tikzpicture}
	\caption{Accuracy of the \ac{rbt}-solver without iterative refinement for the various sizes of the \matname{circul} matrix.}
	\label{fig:exp:acc-gen-bt-struct-circul}
	\Description{A cloud of points between \num{3e-17} and \num{1e-17} representing the backward errors for the non-Parker \ac{rbt} tests.  On either end, the Parker \ac{rbt} tests have similar errors.}
\end{figure}%
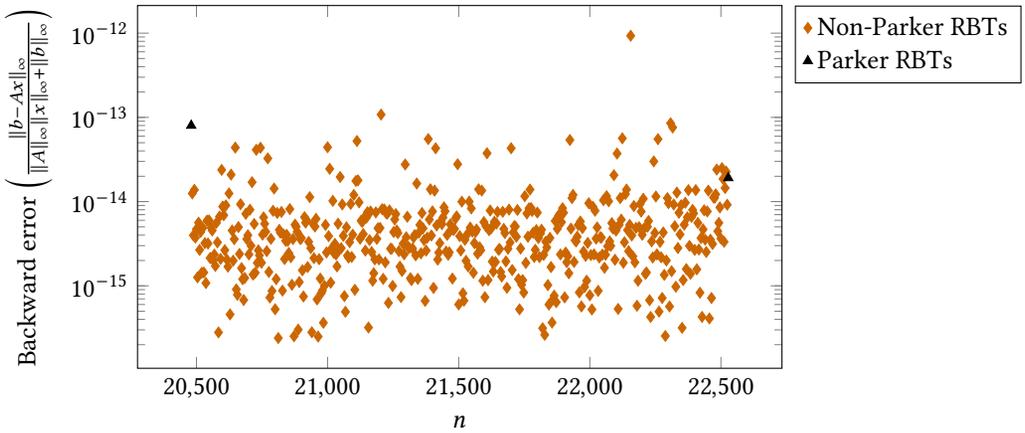
\begin{figure}
	\centering
	\begin{tikzpicture}\begin{semilogyaxis}[width=4in, height=2.5in,ylabel=Backward error \(\left(\frac{\norm{\infty}{b-Ax}}{\norm{\infty}{A}\norm{\infty}{x} + \norm{\infty}{b}}\right)\),xlabel=\(n\),xticklabel style={/pgfplots/scaled x ticks=false},legend style={at={(1.02,1)},anchor=north west},legend cell align=left,]
\addplot[only marks,orange!80!black,mark=diamond*] coordinates {(20484,1.25e-14) (20488,3.97e-15) (20492,1.38e-14) (20496,3.74e-15) (20500,4.71e-15) (20504,1.26e-15) (20508,5.63e-15) (20512,2.67e-15) (20516,4.37e-15) (20520,1.43e-15) (20524,4.87e-15) (20528,1.44e-15) (20532,3.21e-15) (20536,1.08e-15) (20540,5.94e-15) (20544,3.18e-15) (20548,5.98e-15) (20552,2.14e-15) (20556,4.9e-15) (20560,4.54e-15) (20564,2e-15) (20568,2.35e-15) (20572,1.72e-15) (20576,5.59e-15) (20580,3.28e-15) (20584,2.8e-16) (20588,6.65e-15) (20592,2.09e-15) (20596,2.38e-14) (20600,8.73e-15) (20604,6.89e-15) (20608,2.65e-15) (20612,9.02e-15) (20616,2.05e-15) (20620,1.69e-15) (20624,1.25e-14) (20628,4.58e-16) (20632,2.09e-14) (20636,4.54e-15) (20640,1.99e-15) (20644,3.58e-15) (20648,4.38e-14) (20652,9.08e-16) (20656,7.83e-16) (20660,4.98e-15) (20664,2.39e-15) (20668,9.31e-15) (20672,1.62e-15) (20676,1.2e-15) (20680,6.79e-16) (20684,1.24e-15) (20688,7.16e-15) (20692,3.59e-15) (20696,3.99e-15) (20700,2.36e-15) (20704,8.01e-15) (20708,2.99e-15) (20712,1.7e-14) (20716,1.36e-15) (20720,5.43e-15) (20724,1.43e-15) (20728,4.12e-14) (20732,1.87e-15) (20736,2.64e-15) (20740,2.3e-15) (20744,4.36e-14) (20748,1.88e-15) (20752,4.06e-15) (20756,6.5e-15) (20760,5.45e-15) (20764,4.39e-15) (20768,2.58e-15) (20772,3.26e-14) (20776,1.46e-15) (20780,7.27e-16) (20784,3.61e-15) (20788,8.69e-16) (20792,7.78e-15) (20796,1.43e-14) (20800,5.27e-16) (20804,1.2e-15) (20808,7.33e-15) (20812,2.4e-16) (20816,4.61e-15) (20820,4.27e-15) (20824,7.53e-15) (20828,8.77e-16) (20832,4.08e-15) (20836,2.08e-15) (20840,3.43e-15) (20844,4.03e-15) (20848,3.42e-15) (20852,2.45e-15) (20856,8.15e-15) (20860,2.19e-15) (20864,1.03e-15) (20868,3.59e-15) (20872,2.52e-16) (20876,5.33e-15) (20880,5.02e-15) (20884,4.37e-15) (20888,3.02e-16) (20892,1.11e-15) (20896,3.59e-15) (20900,3.69e-15) (20904,1.53e-15) (20908,7.47e-16) (20912,6.54e-15) (20916,5.63e-15) (20920,6.02e-15) (20924,1.67e-15) (20928,3.33e-15) (20932,1.13e-14) (20936,2.8e-15) (20940,2.81e-16) (20944,5.46e-15) (20948,2.13e-15) (20952,5.06e-15) (20956,6.24e-15) (20960,6.9e-16) (20964,2.5e-16) (20968,2.07e-15) (20972,1.22e-15) (20976,8.2e-16) (20980,8.78e-16) (20984,3.64e-16) (20988,9.98e-16) (20992,5.28e-15) (20996,2.49e-15) (21000,4.42e-14) (21004,3.89e-15) (21008,2.45e-14) (21012,1.09e-15) (21016,2.79e-15) (21020,1.02e-14) (21024,2.46e-15) (21028,3.15e-15) (21032,2.36e-15) (21036,1.7e-15) (21040,5.46e-15) (21044,2.67e-15) (21048,1.96e-14) (21052,2.8e-15) (21056,9.25e-15) (21060,3.92e-15) (21064,7.55e-16) (21068,4.92e-16) (21072,4.95e-15) (21076,2.28e-15) (21080,2.17e-15) (21084,3.3e-15) (21088,9.39e-15) (21092,2.2e-15) (21096,4.05e-15) (21100,1.2e-14) (21104,8.98e-16) (21108,1.76e-14) (21112,5.24e-14) (21116,1.79e-14) (21120,9.8e-15) (21124,3.63e-15) (21128,5.94e-15) (21132,5.73e-15) (21136,4.11e-15) (21140,6.48e-15) (21144,2.62e-15) (21148,6.43e-15) (21152,7.48e-15) (21156,3.2e-16) (21160,3.49e-15) (21164,7.64e-15) (21168,2.52e-15) (21172,1.15e-15) (21176,1.41e-15) (21180,2.06e-15) (21184,7.45e-15) (21188,3.61e-15) (21192,4.83e-15) (21196,2.06e-15) (21200,4.86e-15) (21204,1.08e-13) (21208,8.13e-15) (21212,1.21e-15) (21216,8.06e-15) (21220,1.9e-15) (21224,7.25e-15) (21228,7.22e-15) (21232,6.27e-16) (21236,7.87e-15) (21240,2.76e-15) (21244,2.75e-15) (21248,9.96e-16) (21252,4.35e-15) (21256,6.14e-15) (21260,7.11e-15) (21264,4.54e-15) (21268,4.22e-15) (21272,2.58e-15) (21276,1.69e-15) (21280,7.35e-16) (21284,1.5e-15) (21288,3.16e-15) (21292,2.37e-15) (21296,2.76e-14) (21300,3.77e-15) (21304,6.6e-15) (21308,2.88e-15) (21312,2.81e-15) (21316,1.23e-15) (21320,1.17e-15) (21324,2.78e-15) (21328,3.75e-15) (21332,4e-15) (21336,2.3e-15) (21340,1.64e-14) (21344,1.21e-15) (21348,4.28e-15) (21352,5.29e-15) (21356,2.43e-15) (21360,7.32e-15) (21364,4.82e-15) (21368,3.49e-15) (21372,6.63e-16) (21376,2.62e-15) (21380,3.73e-15) (21384,5.53e-14) (21388,6.13e-15) (21392,1.4e-14) (21396,3.96e-15) (21400,7.65e-15) (21404,9.45e-16) (21408,1.35e-14) (21412,4.3e-14) (21416,5.21e-15) (21420,2.05e-15) (21424,2.48e-15) (21428,3.21e-15) (21432,5.83e-15) (21436,5.4e-15) (21440,6.48e-15) (21444,8.55e-15) (21448,7.69e-15) (21452,2.55e-15) (21456,4.02e-15) (21460,2.34e-15) (21464,2.5e-15) (21468,5.37e-15) (21472,4.68e-15) (21476,1.22e-15) (21480,4.6e-15) (21484,4.74e-15) (21488,6.31e-15) (21492,6.67e-15) (21496,2.78e-14) (21500,6.01e-16) (21504,3.89e-15) (21508,2.63e-15) (21512,9.99e-15) (21516,8.15e-16) (21520,6.66e-16) (21524,3.52e-15) (21528,1.98e-15) (21532,4.07e-15) (21536,1.7e-15) (21540,9.99e-15) (21544,4.91e-15) (21548,8.31e-15) (21552,4.41e-15) (21556,1.43e-15) (21560,3.21e-15) (21564,4.35e-15) (21568,4.66e-15) (21572,2.41e-15) (21576,1.4e-14) (21580,2.86e-15) (21584,2.64e-15) (21588,1.35e-14) (21592,1.15e-15) (21596,5.71e-15) (21600,7.33e-16) (21604,5.1e-15) (21608,3.75e-14) (21612,1.17e-15) (21616,4.57e-15) (21620,7.84e-15) (21624,4.83e-15) (21628,1.94e-15) (21632,1.48e-15) (21636,5.76e-15) (21640,1.92e-15) (21644,5.5e-15) (21648,7.55e-15) (21652,2.68e-15) (21656,2.74e-15) (21660,1.46e-15) (21664,1.29e-15) (21668,7.56e-15) (21672,4.74e-15) (21676,4.23e-15) (21680,1.95e-15) (21684,9.48e-16) (21688,2.76e-15) (21692,8.03e-15) (21696,3.27e-15) (21700,4.3e-14) (21704,1.54e-15) (21708,4.3e-15) (21712,5.1e-15) (21716,5.06e-15) (21720,3.83e-15) (21724,7.95e-15) (21728,1.2e-15) (21732,5.26e-16) (21736,2.77e-15) (21740,1.03e-15) (21744,1.17e-15) (21748,8.32e-16) (21752,1.16e-14) (21756,1.87e-15) (21760,7.23e-15) (21764,3.83e-15) (21768,8.32e-15) (21772,1.39e-14) (21776,8.38e-15) (21780,7.74e-15) (21784,4.27e-15) (21788,2.2e-15) (21792,2.68e-15) (21796,4.58e-15) (21800,7.41e-15) (21804,3.69e-15) (21808,2.19e-15) (21812,4.19e-15) (21816,4.63e-15) (21820,3.14e-16) (21824,2.33e-15) (21828,2.62e-16) (21832,3.34e-15) (21836,2.86e-15) (21840,1.02e-15) (21844,5.99e-16) (21848,3.7e-15) (21852,1.67e-15) (21856,3.66e-16) (21860,6.93e-16) (21864,7.64e-16) (21868,5.9e-15) (21872,6.4e-16) (21876,4.75e-15) (21880,6.7e-15) (21884,7.06e-15) (21888,2.36e-15) (21892,5.7e-15) (21896,8.35e-15) (21900,7.38e-15) (21904,7.29e-16) (21908,2.48e-15) (21912,1.91e-15) (21916,2.25e-15) (21920,1.41e-15) (21924,5.39e-14) (21928,4.78e-15) (21932,1.1e-14) (21936,1.2e-14) (21940,1.92e-15) (21944,3.46e-15) (21948,3.37e-15) (21952,3.83e-15) (21956,5.27e-15) (21960,8.35e-16) (21964,5.8e-15) (21968,5.69e-16) (21972,4.85e-15) (21976,2.55e-15) (21980,9.69e-16) (21984,2.27e-15) (21988,1.15e-15) (21992,3.23e-15) (21996,9.57e-15) (22000,7.72e-16) (22004,2.32e-15) (22008,5.24e-16) (22012,8.67e-15) (22016,2.36e-15) (22020,7.79e-15) (22024,3.64e-15) (22028,4.93e-15) (22032,8.72e-16) (22036,2.17e-15) (22040,3.77e-15) (22044,3.74e-15) (22048,2.35e-15) (22052,1.04e-14) (22056,1.64e-15) (22060,2.34e-15) (22064,1.16e-14) (22068,2.95e-15) (22072,1.01e-14) (22076,3.37e-15) (22080,2.98e-15) (22084,5.55e-15) (22088,4.79e-15) (22092,2.06e-14) (22096,2.04e-15) (22100,4.53e-15) (22104,3.72e-14) (22108,5.32e-16) (22112,5.3e-15) (22116,1.1e-14) (22120,6.1e-15) (22124,5.64e-14) (22128,1.88e-15) (22132,3.29e-15) (22136,1.23e-14) (22140,1.39e-14) (22144,4.59e-15) (22148,2.58e-15) (22152,5.29e-15) (22156,9.31e-13) (22160,3.96e-15) (22164,2.24e-15) (22168,8.85e-15) (22172,1.9e-15) (22176,4.87e-15) (22180,5.8e-16) (22184,1.15e-15) (22188,4.97e-15) (22192,9.37e-15) (22196,7.62e-15) (22200,4.04e-15) (22204,1.08e-14) (22208,4.84e-15) (22212,1.39e-15) (22216,1.44e-15) (22220,5.03e-15) (22224,5.88e-15) (22228,6.75e-16) (22232,4.27e-16) (22236,3.37e-15) (22240,2.17e-15) (22244,3e-14) (22248,4.23e-15) (22252,1.15e-14) (22256,2.36e-15) (22260,5.52e-14) (22264,4.94e-16) (22268,3.33e-15) (22272,2.18e-15) (22276,7.18e-16) (22280,4.3e-15) (22284,1.35e-15) (22288,2.54e-16) (22292,2e-15) (22296,1.1e-14) (22300,9.58e-15) (22304,5.69e-16) (22308,8.52e-14) (22312,1.49e-15) (22316,7.59e-14) (22320,9.24e-15) (22324,5.1e-15) (22328,9.76e-16) (22332,1.28e-14) (22336,2.81e-15) (22340,7.37e-15) (22344,4.08e-15) (22348,9.05e-15) (22352,3.17e-16) (22356,1.17e-15) (22360,5.32e-15) (22364,9.69e-15) (22368,1.58e-15) (22372,6.5e-15) (22376,1.37e-14) (22380,2.98e-15) (22384,1.39e-15) (22388,6.58e-15) (22392,7.97e-15) (22396,7.25e-15) (22400,2.67e-15) (22404,1.37e-14) (22408,1.57e-15) (22412,5.58e-15) (22416,2.7e-15) (22420,8.29e-16) (22424,3.69e-15) (22428,4.28e-16) (22432,4.68e-15) (22436,7.89e-15) (22440,3.54e-15) (22444,2.52e-15) (22448,7.64e-15) (22452,1.35e-14) (22456,4.1e-16) (22460,3.53e-15) (22464,7.17e-16) (22468,3.14e-15) (22472,1.12e-14) (22476,1.11e-14) (22480,1.22e-14) (22484,2.4e-14) (22488,4.45e-15) (22492,5.58e-15) (22496,8.54e-15) (22500,3.64e-15) (22504,2.5e-14) (22508,1.86e-14) (22512,3.33e-15) (22516,1.45e-14) (22520,2.27e-14) (22524,9.18e-15)};
\addlegendentry{Non-Parker \acsp{rbt}};
\addplot[only marks,black,mark=triangle*] coordinates {(20480,8e-14) (22528,1.91e-14)};
\addlegendentry{Parker \acsp{rbt}};
\end{semilogyaxis} \end{tikzpicture}
	\caption{Accuracy of the \ac{rbt}-solver without iterative refinement for the various sizes of the \matname{chebspec} matrix.}
	\label{fig:exp:acc-gen-bt-struct-chebspec}
	\Description{A cloud of points between \num{1e-16} and \num{1e-13} (with a few outliers) representing the backward errors for the non-Parker \ac{rbt} tests.  The Parker \ac{rbt} tests have errors of \num{8e-14} and \num{2e-14} on the left and right sides, respectively.}
\end{figure}%
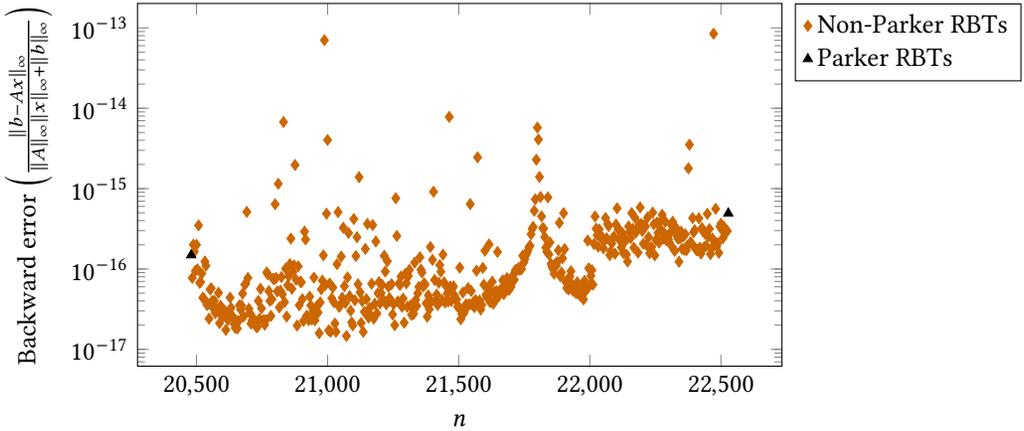
\begin{figure}
	\centering
	\begin{tikzpicture}\begin{semilogyaxis}[width=4in, height=2.5in,ylabel=Backward error \(\left(\frac{\norm{\infty}{b-Ax}}{\norm{\infty}{A}\norm{\infty}{x} + \norm{\infty}{b}}\right)\),xlabel=\(n\),xticklabel style={/pgfplots/scaled x ticks=false},legend style={at={(1.02,1)},anchor=north west},legend cell align=left,]
\addplot[only marks,orange!80!black,mark=diamond*] coordinates {(20484,7.75e-17) (20488,2.01e-16) (20492,1.64e-16) (20496,9.49e-17) (20500,2e-16) (20504,9.91e-17) (20508,3.49e-16) (20512,6.73e-17) (20516,6.88e-17) (20520,8.76e-17) (20524,4.39e-17) (20528,4.3e-17) (20532,1.24e-16) (20536,1.1e-16) (20540,3.57e-17) (20544,3.68e-17) (20548,2.41e-17) (20552,5.56e-17) (20556,5.76e-17) (20560,4.01e-17) (20564,3.25e-17) (20568,4.03e-17) (20572,2.82e-17) (20576,2.99e-17) (20580,3.18e-17) (20584,5.05e-17) (20588,2.12e-17) (20592,3.49e-17) (20596,2.62e-17) (20600,2.8e-17) (20604,2.58e-17) (20608,3.01e-17) (20612,1.75e-17) (20616,2.82e-17) (20620,2.33e-17) (20624,2.57e-17) (20628,4.41e-17) (20632,3.29e-17) (20636,2.6e-17) (20640,2.48e-17) (20644,2.07e-17) (20648,1.85e-17) (20652,2.27e-17) (20656,1.82e-17) (20660,2.24e-17) (20664,3.53e-17) (20668,3.14e-17) (20672,3.51e-17) (20676,4.94e-17) (20680,2.57e-17) (20684,6.39e-17) (20688,6.6e-17) (20692,5.14e-16) (20696,7.74e-17) (20700,3.2e-17) (20704,2.82e-17) (20708,2.09e-17) (20712,2.29e-17) (20716,2.45e-17) (20720,2.25e-17) (20724,2.25e-17) (20728,1.91e-17) (20732,1.89e-17) (20736,4.01e-17) (20740,2.4e-17) (20744,5.67e-17) (20748,3.24e-17) (20752,2.26e-17) (20756,2.3e-17) (20760,2.03e-17) (20764,2.31e-17) (20768,9.06e-17) (20772,8.58e-17) (20776,5.31e-17) (20780,4.02e-17) (20784,2.61e-17) (20788,4.2e-17) (20792,7.57e-17) (20796,7.47e-17) (20800,6.4e-16) (20804,7.75e-17) (20808,5.26e-17) (20812,1.15e-15) (20816,3.33e-17) (20820,5.37e-17) (20824,7.94e-17) (20828,4.86e-17) (20832,6.74e-15) (20836,3.87e-17) (20840,1.01e-16) (20844,2.56e-17) (20848,6.28e-17) (20852,8.13e-17) (20856,1.15e-16) (20860,2.39e-16) (20864,5.67e-17) (20868,9.17e-17) (20872,1.11e-16) (20876,1.97e-15) (20880,7.36e-17) (20884,1.98e-17) (20888,1.08e-16) (20892,3.54e-17) (20896,6.95e-17) (20900,6.84e-17) (20904,4.21e-17) (20908,2.26e-17) (20912,2.92e-16) (20916,2.33e-16) (20920,2.62e-17) (20924,2.53e-17) (20928,2.31e-17) (20932,4.53e-17) (20936,3.52e-17) (20940,2.36e-17) (20944,2.98e-17) (20948,7.81e-17) (20952,2.3e-17) (20956,2.88e-17) (20960,3.49e-17) (20964,2.88e-17) (20968,1.59e-17) (20972,3.84e-17) (20976,5.59e-17) (20980,7.18e-17) (20984,1.48e-16) (20988,7.03e-14) (20992,6.35e-17) (20996,4.88e-16) (21000,4.03e-15) (21004,1.72e-17) (21008,6.6e-17) (21012,1.66e-17) (21016,4.32e-17) (21020,5.88e-17) (21024,3.87e-17) (21028,2.1e-17) (21032,1.66e-17) (21036,1.15e-16) (21040,5.11e-16) (21044,3.39e-17) (21048,3.5e-17) (21052,1.43e-16) (21056,4.4e-17) (21060,3.26e-16) (21064,6.11e-17) (21068,3.78e-17) (21072,1.47e-17) (21076,7.87e-17) (21080,2.81e-16) (21084,4.6e-17) (21088,3.03e-17) (21092,2.17e-17) (21096,1.99e-17) (21100,4.19e-16) (21104,4.08e-17) (21108,1.45e-16) (21112,2.49e-16) (21116,6.47e-17) (21120,1.39e-15) (21124,5.13e-17) (21128,2.13e-17) (21132,3.98e-17) (21136,1.65e-17) (21140,2.93e-17) (21144,1.78e-16) (21148,4.22e-17) (21152,3.64e-16) (21156,2.49e-17) (21160,3.27e-17) (21164,2.67e-17) (21168,6.55e-17) (21172,3.53e-16) (21176,2.42e-17) (21180,3.69e-17) (21184,2.2e-16) (21188,2.8e-17) (21192,5.01e-17) (21196,6.87e-17) (21200,3.74e-17) (21204,5.01e-17) (21208,6.36e-17) (21212,4.16e-17) (21216,1.46e-16) (21220,9.31e-17) (21224,3.64e-17) (21228,1.25e-16) (21232,6.02e-17) (21236,4.35e-17) (21240,4e-17) (21244,6.48e-17) (21248,2.72e-17) (21252,3.66e-17) (21256,2e-17) (21260,7.6e-16) (21264,2.58e-16) (21268,3.95e-17) (21272,6.62e-17) (21276,6.33e-17) (21280,8.07e-17) (21284,3.39e-17) (21288,4.28e-17) (21292,2.15e-17) (21296,8.67e-17) (21300,6.63e-17) (21304,3.39e-17) (21308,8.2e-17) (21312,9.77e-17) (21316,3.52e-17) (21320,2.69e-17) (21324,3.37e-17) (21328,2.76e-17) (21332,1.2e-16) (21336,4.12e-17) (21340,3.47e-17) (21344,4.98e-17) (21348,5.98e-17) (21352,1.4e-16) (21356,9.82e-17) (21360,9.05e-17) (21364,6.09e-17) (21368,6.65e-17) (21372,6.42e-17) (21376,3.74e-17) (21380,2.55e-17) (21384,4.7e-17) (21388,3.15e-17) (21392,1.88e-16) (21396,3.94e-17) (21400,1.3e-16) (21404,9.15e-16) (21408,7.32e-17) (21412,4.12e-17) (21416,4.81e-17) (21420,3.56e-17) (21424,6.74e-17) (21428,1.12e-16) (21432,9.25e-17) (21436,4.12e-17) (21440,1.51e-16) (21444,6.84e-17) (21448,3.83e-17) (21452,6.12e-17) (21456,4.95e-17) (21460,3.95e-17) (21464,7.8e-15) (21468,6.05e-17) (21472,4.01e-17) (21476,5.71e-17) (21480,6.1e-17) (21484,4.69e-17) (21488,4.14e-17) (21492,4.18e-17) (21496,3.39e-17) (21500,3.42e-17) (21504,1.04e-16) (21508,2.37e-17) (21512,2.74e-17) (21516,5.1e-17) (21520,2.82e-17) (21524,3.09e-17) (21528,3.63e-17) (21532,3.48e-17) (21536,5.78e-17) (21540,3.64e-17) (21544,6.42e-16) (21548,5.33e-17) (21552,3.56e-17) (21556,3.85e-17) (21560,3.36e-17) (21564,4.11e-17) (21568,9.52e-17) (21572,2.45e-15) (21576,3.9e-17) (21580,4.05e-17) (21584,3.14e-17) (21588,3.56e-17) (21592,1.03e-16) (21596,6.37e-17) (21600,1.69e-16) (21604,4.78e-17) (21608,4.42e-17) (21612,5.84e-17) (21616,2.02e-16) (21620,3.86e-17) (21624,5.07e-17) (21628,6.95e-17) (21632,3.7e-17) (21636,5.16e-17) (21640,4.11e-17) (21644,5.31e-17) (21648,1.63e-16) (21652,4.57e-17) (21656,5.33e-17) (21660,5.54e-17) (21664,5.23e-17) (21668,7.22e-17) (21672,4.87e-17) (21676,6.92e-17) (21680,6.17e-17) (21684,7.7e-17) (21688,7.11e-17) (21692,5.5e-17) (21696,6.36e-17) (21700,6.01e-17) (21704,8.08e-17) (21708,7.17e-17) (21712,7.36e-17) (21716,9.67e-17) (21720,9.7e-17) (21724,9.79e-17) (21728,1.03e-16) (21732,9.82e-17) (21736,9.9e-17) (21740,9.99e-17) (21744,1.46e-16) (21748,1.2e-16) (21752,1.32e-16) (21756,1.48e-16) (21760,1.58e-16) (21764,1.7e-16) (21768,1.87e-16) (21772,1.95e-16) (21776,2.71e-16) (21780,3.16e-16) (21784,3.35e-16) (21788,5.3e-16) (21792,7.37e-16) (21796,2.29e-15) (21800,5.75e-15) (21804,4.09e-15) (21808,1.4e-15) (21812,7.87e-16) (21816,4.52e-16) (21820,3.22e-16) (21824,3.19e-16) (21828,2.5e-16) (21832,1.85e-16) (21836,1.64e-16) (21840,7.79e-16) (21844,1.52e-16) (21848,2.14e-16) (21852,1.16e-16) (21856,1.06e-16) (21860,9.94e-17) (21864,8.74e-17) (21868,1.16e-16) (21872,1.1e-16) (21876,1.22e-16) (21880,1.71e-16) (21884,3.75e-16) (21888,8.11e-17) (21892,9.06e-17) (21896,1.36e-16) (21900,4.96e-16) (21904,1.77e-16) (21908,6.93e-17) (21912,5.78e-17) (21916,5.82e-17) (21920,5.58e-17) (21924,8.11e-17) (21928,9.39e-17) (21932,6.46e-17) (21936,5.36e-17) (21940,5.39e-17) (21944,5.96e-17) (21948,4.98e-17) (21952,6.83e-17) (21956,5.29e-17) (21960,4.78e-17) (21964,5.04e-17) (21968,4.52e-17) (21972,6.01e-17) (21976,4.19e-17) (21980,7.13e-17) (21984,5.32e-17) (21988,5.78e-17) (21992,2.26e-16) (21996,1.02e-16) (22000,6.26e-17) (22004,8.22e-17) (22008,9.67e-17) (22012,6.39e-17) (22016,2.05e-16) (22020,4.51e-16) (22024,3.17e-16) (22028,2.63e-16) (22032,1.9e-16) (22036,3.11e-16) (22040,2.68e-16) (22044,2.45e-16) (22048,2.35e-16) (22052,1.81e-16) (22056,2.27e-16) (22060,4.31e-16) (22064,2.42e-16) (22068,3.63e-16) (22072,3.43e-16) (22076,1.49e-16) (22080,1.67e-16) (22084,2.08e-16) (22088,2.4e-16) (22092,2.54e-16) (22096,1.55e-16) (22100,3.58e-16) (22104,5.69e-16) (22108,3.14e-16) (22112,2.33e-16) (22116,2.98e-16) (22120,1.86e-16) (22124,2.81e-16) (22128,1.58e-16) (22132,2.48e-16) (22136,2.24e-16) (22140,1.51e-16) (22144,1.24e-16) (22148,3.39e-16) (22152,2.91e-16) (22156,2.43e-16) (22160,4.97e-16) (22164,3.16e-16) (22168,1.84e-16) (22172,4.36e-16) (22176,2.85e-16) (22180,3.09e-16) (22184,1.6e-16) (22188,3.57e-16) (22192,5.84e-16) (22196,3.64e-16) (22200,2.98e-16) (22204,2.06e-16) (22208,3.51e-16) (22212,3.36e-16) (22216,2e-16) (22220,4.37e-16) (22224,2.49e-16) (22228,1.57e-16) (22232,1.85e-16) (22236,1.86e-16) (22240,5.04e-16) (22244,3.6e-16) (22248,1.48e-16) (22252,2.3e-16) (22256,3.78e-16) (22260,2.81e-16) (22264,2.13e-16) (22268,3.88e-16) (22272,3.08e-16) (22276,4.74e-16) (22280,2.39e-16) (22284,2.8e-16) (22288,3.88e-16) (22292,2.27e-16) (22296,2.48e-16) (22300,2.32e-16) (22304,2.11e-16) (22308,2.88e-16) (22312,3.48e-16) (22316,3.88e-16) (22320,1.77e-16) (22324,1.68e-16) (22328,3.28e-16) (22332,2.58e-16) (22336,2.38e-16) (22340,1.23e-16) (22344,2.42e-16) (22348,2.91e-16) (22352,1.78e-16) (22356,2.53e-16) (22360,1.74e-16) (22364,2.71e-16) (22368,1.71e-16) (22372,1.96e-16) (22376,1.79e-15) (22380,3.52e-15) (22384,2.15e-16) (22388,2.08e-16) (22392,4.31e-16) (22396,2.86e-16) (22400,3.05e-16) (22404,1.78e-16) (22408,1.71e-16) (22412,1.74e-16) (22416,4.27e-16) (22420,1.84e-16) (22424,1.56e-16) (22428,3.24e-16) (22432,2.94e-16) (22436,1.87e-16) (22440,2.74e-16) (22444,1.94e-16) (22448,4.86e-16) (22452,2.56e-16) (22456,3.15e-16) (22460,1.53e-16) (22464,1.79e-16) (22468,2.39e-16) (22472,8.47e-14) (22476,2.17e-16) (22480,5.61e-16) (22484,1.92e-16) (22488,2.09e-16) (22492,1.59e-16) (22496,3.39e-16) (22500,3.65e-16) (22504,3.07e-16) (22508,2.43e-16) (22512,3.34e-16) (22516,2.77e-16) (22520,3.09e-16) (22524,2.96e-16)};
\addlegendentry{Non-Parker \acsp{rbt}};
\addplot[only marks,black,mark=triangle*] coordinates {(20480,1.49e-16) (22528,4.94e-16)};
\addlegendentry{Parker \acsp{rbt}};
\end{semilogyaxis} \end{tikzpicture}
	\caption{Accuracy of the \ac{rbt}-solver without iterative refinement for the various sizes of the \matname{fiedler} matrix.}
	\label{fig:exp:acc-gen-bt-struct-fiedler}
	\Description{A cloud of points between \num{1e-17} and \num{1e-15} (with some outliers) representing the backward errors for the non-Parker \ac{rbt} tests.  The band seems to show some trends, unlike the plots for \matname{circul} and \matname{chebspec}.  The Parker \ac{rbt} tests have errors of \num{1e-16} and \num{5e-16} on the left and right sides, respectively.}
\end{figure}%
These matrices are cases where the Parker \ac{rbt} can improve their accuracy, even without iterative refinement.
The first and last problems are multiples of 2048 (i.e., 4 times the tile size), giving the strict recursive structure of Parker \acp{rbt}.
Then, the generalized formula was tested with every intermediate multiple of four.
This latter group represents all of the cases where a Parker \ac{rbt} could be applied but our generalization has a different structure.
Because we are interested specifically in the factorization error, no iterative refinement is applied in this test.
So, if the generalized structure was problematically worse than Parker's formulation, the Parker cases would be noticeably better than the non-Parker cases  (particularly the right side).
However, most of the non-Parker cases are similar or better than the Parker cases.
First, for \matname{circul}, all the errors varied by less than a factor of two, demonstrating little sensitivity to the \ac{rbt} structure.
Next, for \matname{chebspec}, the errors varied significantly more without an obvious pattern.
But, \SI{94}{\%} of the non-Parker tests had a smaller error than both of the Parker tests, and only three of the non-Parker tests had a larger error than both of the Parker tests.
Finally, \matname{fiedler} is the most intriguing with visible trends in the data.
Particularly, the error initially decreases from the left Parker \ac{rbt}, spikes at around \(n=\num{21880}\), then jumps up by half a digit around \(n=\num{22016}\) (although there are numerous outliers, particularly on the left).
Note that \(n=\num{22016}\) corresponds to exactly 43 tile rows.
Interestingly, the other two tile boundaries (\(n=\num{20992}, \num{21504}\)) do not have similar jumps in the accuracy.
Fortunately, \SI{63}{\%} of the non-Parker cases had a smaller error than both of the Parker cases, and only \SI{6}{\%} were larger than both.
So, while our generalization seems to have some form of effect, it is not a significant overall reduction in stability.

\subsection{Failures of the \acs{rbt} Solver}
\label{sec:exp:orr}

\Cref{tab:acc:back-error} shows that the \ac{rbt} solver matched the accuracy of \ac{gepp} in all but three matrices: \matname{orthog}, \matname{ris}, and \matname{riemann}.
Understanding these problems is important for understanding the weaknesses of the approach.
Note that large element growth can only occur when one or more leading principal submatrices have small singular values (see \eqref{eq:growth-bound}).
The \matname{ris} matrix is the most straightforward; its entries are all close to zero except for a band along the anti-diagonal.
Thus, the \ac{rbt} needs to apply the anti-diagonal band to the first diagonal element,
which requires a depth close to \(\log_2(n)\).
This behavior was distilled to \cref{eq:example-for-small-d} as an example of how a depth less than \(\ceil{\log_2(n)}\) is insufficient for strong nonsingularity.

The \matname{orthog} matrix has a history of being challenging for \ac{gepp} alternatives~\cite{beckerReducingAmountPivoting2012,donfackSurveyRecentDevelopments2015,panNumericallySafeGaussian2017,lindquistReplacingPivotingDistributed2020,lindquistThresholdPivotingDense2022,lindquistUsingAdditiveModifications2023}.
This matrix is defined as
\[
A[i,j] =
\sqrt{\frac{2}{n+1}}\sin\Bigl(\frac{i \times j \times \pi}{n+1}\Bigr).
\]
When \(i \times j \ll n\), sine is approximately linear in its argument.
Thus, the first few leading principal submatrices are close to rank-1 matrices, so plain \ac{genp} will have a very large growth factor.
While the \acp{rbt} combine \(2^{2d}\) submatrices to form the leading principal submatrix, it turns out that these submatrices are all numerically low rank.
For the \matname{orthog} matrix in \cref{tab:acc:back-error}, the 16 submatrices that combine to form the leading \(512\times512\) principal submatrix, together have only 266 singular values larger than \(10^{-11}\) (with \(6\times 10^{-2}\) being the largest element in the original matrix).\footnote{\label{foot:submatrix-tests}See \texttt{submatrix-tests.jl} in the code artifact for more details.}
Furthermore, after applying the first orthogonal transform of the \ac{rbt} to each side, the resulting four matrices together have only 147 singular values larger than \(10^{-11}\) (if a Parker RBT is used instead, this decreases to 142 singular values).
Thus, the leading principal block has many small singular values.
This is fundamentally the same issue as in \matname{ris} or \cref{eq:example-for-small-d}, except the rank deficiency is less obvious.
%

Finally, \matname{riemann} is both the most concerning (since \ac{rbt}-preprocessing reduces the accuracy, particularly the generalized formulation) and the most enigmatic.
This matrix is related to the Riemann hypothesis and integer divisibility~\cite{roeslerRiemannHypothesisEigenvalue1986}.
%
%
%
Given the accuracy of \ac{genp} on this matrix and additional experimental results,\textsuperscript{\ref{foot:submatrix-tests}} the leading principal submatrices of \matname{riemann}  appear full rank.
However, if we apply a two-sided \ac{rbt} where the random variables are set to 1, the resulting \(256\times256\) leading principal submatrix has only 240 singular values larger than \(10^{-8}\) (with \(1.5\times 10^{5}\) being the largest element in the original matrix).
The random variables are supposed to prevent these types of cancellations, but this suggests that a larger range of random values may be useful for some matrices.
Interestingly, although iterative refinement could converge for the Parker \ac{rbt} but not the semi-Parker \ac{rbt}, using a Parker \ac{rbt} results in only 235 singular values larger than \(10^{-8}\) for the leading \(256\times 256\) block.
This suggests that there are likely additional factors limiting the accuracy of the \ac{rbt}-solvers on this matrix, possibly relating to the singular- or eigen-vectors.

\subsection{Performance}

Figure~\ref{fig:exp:perf} compares the performance of the generalized \ac{rbt} solver with that of \ac{gepp}, \ac{genp}, and the old formulation of \acp{rbt}.
The performance tests used \matname{rand}, except for the diagonally dominant cases for \ac{gepp} \ac{getp} which used \matname{rand+nI}.
Each configuration was run four times, and the runtimes were summarized with the mean and the \SI{99}{\%} \ac{ci};
the performance was then computed as a \(\tfrac{2}{3}n^3\) \si{\flop} count
divided by either the mean, the lower \ac{ci} bound, or the upper
\ac{ci} bound in seconds.
\begin{figure}
	\centering
	\input{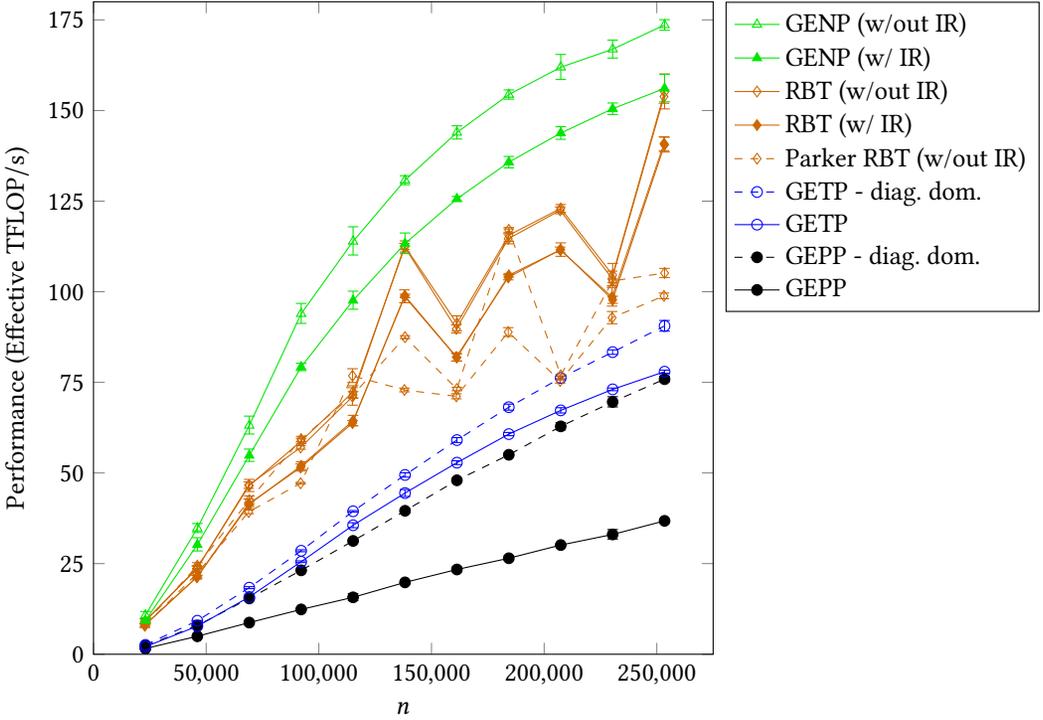}
	\caption{Performance of the \ac{rbt} solver versus \ac{gepp} and
        \ac{genp}.  The performance is computed based on
        \(\tfrac{2}{3}n^3\) \si{\flop}, to normalize the performance
        results.  Error bars indicate the \SI{99}{\%} confidence intervals based on the runtime.}
	\Description{\Ac{genp} (without IR) smoothly increases in performance from about \SI{10}{\tera\flop\per\second} up to about \SI{175}{\tera\flop\per\second} in a slightly convex curve.  \Ac{genp} (with IR) follows a similar but lower curve, reaching about \SI{155}{\tera\flop\per\second}.  \Ac{getp} and \ac{gepp} both start at around \SI{2}{\tera\flop\per\second} and smoothly increases performance as the problem size increases.  \Ac{getp} reaches \num{80} and \SI{90}{\tera\flop\per\second} for a random matrix and a diagonal dominant matrix, respectively.  \Ac{gepp} reaches \num{37} and \num{75}{\tera\flop\per\second} for those same matrices.  The generalized \ac{rbt} (without IR) has a very jagged performance curve, peaking at about the performance of \ac{genp} with IR.  The two sets of problem sizes are almost perfectly overlapped.  Adding IR to the \ac{rbt} solver reduces performance but keeps the shape of the curve.  The Parker \ac{rbt} also has a jagged performance curve, but averaging much lower and having very different curves for the two sets of problem sizes.}
	\label{fig:exp:perf}
\end{figure}
All of the solvers were tested with matrices of dimension \(\num{23040}\times i\) for \(1 \leq i \leq 11\).
Because \(\num{23040}=45\times512\), the Parker \ac{rbt} will be aligned to whole tiles only every fourth matrix size.
Additionally, the \ac{rbt} solvers were tested with matrices smaller by half of a tile, which always corresponds to cases in which the Parker \ac{rbt} solver is not aligned to whole tiles.
We did not include a set of matrices where the Parker \ac{rbt} is always
aligned since, in that case, the generalized and Parker \acp{rbt} are
the same structurally.
\Ac{genp} and the \ac{rbt} solvers were tested both with and without iterative refinement.
For clarity's sake, the results for the Parker \ac{rbt} with iterative
refinement are omitted from \cref{fig:exp:perf} but can be found in our
results' artifact.

\Ac{genp} and the \ac{rbt} solvers provided noticeably higher performance than \ac{gepp} and \ac{getp}, although the \ac{rbt} solver was about \SI{30}{\%} slower than \ac{genp}.
For the largest two problem sizes, the \ac{rbt} solver with iterative refinement was \SI{40}{\%} and \SI{85}{\%} faster than the best case of \ac{gepp}.
These speedups double when extensive pivoting is required.
Thus, even if the \ac{rbt} solver must fall back to \ac{gepp} \SI{20}{\%} of the time, it will be on average more than \SI{25}{\%} faster.

When comparing the two \ac{rbt} formulations, the generalized \ac{rbt} outperforms the Parker \ac{rbt} by up to \SI{62}{\%} and only underperforms on one problem size.
However, for problem sizes like \(n=\num{184320}\) where the Parker \ac{rbt} aligns to the tiles, the two formulations provide identical performance.
Because the generalized \ac{rbt} always aligns itself to whole tiles, it performs identically between the two sets of problem sizes.
On the other hand, the Parker \ac{rbt} saw up to a \SI{32}{\%} increase in runtime when decreasing the matrix size by 256 rows.
Despite better tile alignment, the performance of the generalized \ac{rbt} does not change smoothly with the problem size;
the variation depends on whether communication is needed for one or both of the layers in the transform.
Because this test uses a \(4\times 4\) process grid, no communication is needed if and only if \(\ceil{n_t2^{-d}}\) is divisible by \(4\).
In \cref{fig:exp:perf}, this condition occurs for \(n=\num{23040},
n=\num{138240}, n=\num{253440}\), corresponding to the large spikes
in the performance.
Similarly, only one layer of the \ac{rbt} communicates for
\(n=\num{69120}, n=\num{184320}, n=\num{207360}\), corresponding
to the middle band of performance results.
This communication overhead results in a \SI{48}{\%} increase in the \si{\giga\flop\per\second} rate, compared to only a \SI{4}{\%} increase in that metric for \ac{genp};
this corresponds to an \SI{10}{\%} \emph{decrease} in runtime when increasing the problem size by \SI{10}{\%} (and the \si{\flop} count by \SI{33}{\%}).
This suggests that it is preferable to use a butterfly structure where the top half of the first layer  (i.e., the layer with a single butterfly matrix) has dimension
\begin{equation}
	\label{eq:comm-free-rbt-size}
	\mathrm{lcm}(p, q)2^{d-1}\ceil{n_t2^{-d}\mathrm{lcm}(p, q)^{-1}}
\end{equation}
for a \(p\times q\) process grid and \(\mathrm{lcm}(p, q)\) being the least common multiple of \(p\) and \(q\).

Both the \acp{rbt} and iterative refinement added a noticeable overhead, \SI{12}{\%} and \SI{11}{\%} respectively when \(n=\num{253440}\); the former increases to \SI{60}{\%} for the communication intensive \(n=\num{230400}\).
When compared to Parker \acp{rbt}, the new generalized transforms decreased the runtime of the \emph{overall solve} by up to \SI{38}{\%};
it did introduce a slowdown in two cases (\(n=\num{46080}\) and
\(n = \num{92160}\)) but these slowdowns were less than \SI{10}{\%}.

As discussed in \cref{sec:recovering-acc}, when iterative refinement cannot recover sufficient accuracy the problem must be re-solved with a more robust solver.
If we assume the \ac{rbt}-solver fails for one in three matrices (which is worse than \cref{tab:acc:back-error}), our performance results indicate that the \ac{rbt}-solver will outperform \ac{gepp} on \emph{average}.\footnote{Details are provided in our data artifact.}
Furthermore, if the matrices are factored by \ac{gepp} in a time close to that of \matname{rand}, the \ac{rbt} solver will have an average speedup between \SI{54}{\%} and \SI{75}{\%}.
Tournament pivoting is more competitive; when both layers of the \ac{rbt} need to communicate between nodes, the \ac{rbt} solver is \SI{3}{\%} to \SI{12}{\%} slower.
But, if \ac{rbt} does not require inter-node communication, the \ac{rbt} solver is \SI{9}{\%} to \SI{20}{\%} faster.

\subsection{Strong Scaling}

In some applications, the linear systems are small relative to the number of nodes needed for the rest of the application~\cite{ghyselsHighPerformanceSparse2022}.
Thus, strong scaling is important for achieving good performance in those applications.
Towards that end, we tested the \ac{rbt} approach for problems of size \(n=\num{100000}\) with varying numbers of processes, shown in \cref{fig:exp:perf:strong-100000}.
\begin{figure}
	\begin{tikzpicture}\begin{loglogaxis}[width=3.75in, height=3.25in,ylabel=Time (\si{\second}),ytick={5, 10, 20, 40, 80, 160},ymin=5, ymax=160,error bars/y dir=both, error bars/y explicit,yticklabel style={/pgf/number format/fixed, /pgfplots/scaled y ticks=false},xlabel=Number of nodes,xtick={1, 2, 4, 8, 16},xticklabel style={/pgf/number format/fixed, /pgfplots/scaled x ticks=false},log ticks with fixed point,legend style={at={(1.02,1.00)},anchor=north west},legend cell align=left,]
\addplot[black,mark=*] coordinates {(1.0,95.223)-=(0,1.8931185095066354)+=(0,1.8931185095066354) (2.0,122.979)-=(0,11.003570628732874)+=(0,11.003570628732874) (4.0,71.37325)-=(0,4.919753773714035)+=(0,4.919753773714035) (8.0,49.730000000000004)-=(0,4.237645109822047)+=(0,4.237645109822047) (16.0,36.7345)-=(0,0.5539608008750534)+=(0,0.5539608008750534)};
\addlegendentry{\acs{gepp}};
\addplot[black,dotted,forget plot] coordinates {(1.0, 95.223) (32.0, 2.97571875)};
\addplot[black,mark=*,dashed,mark options={solid},error bars/error bar style={solid}] coordinates {(1.0,63.34925)-=(0,3.649523321455959)+=(0,3.649523321455959) (2.0,41.84425)-=(0,2.073894851781958)+=(0,2.073894851781958) (4.0,36.36975)-=(0,0.29967882672745816)+=(0,0.29967882672745816) (8.0,26.0965)-=(0,1.1697031563628464)+=(0,1.1697031563628464) (16.0,25.4255)-=(0,0.2102389305155299)+=(0,0.2102389305155299)};
\addlegendentry{\acs{gepp} - diag. dom.};
\addplot[black,dotted,forget plot] coordinates {(1.0, 63.34925) (32.0, 1.9796640625)};
\addplot[blue,mark=o] coordinates {(2.0,45.27825)-=(0,0.5714167618023183)+=(0,0.5714167618023183) (4.0,26.701999999999998)-=(0,0.3653487599749461)+=(0,0.3653487599749461) (8.0,23.18575)-=(0,0.11599696370233659)+=(0,0.11599696370233659) (16.0,22.823)-=(0,0.12182182205939895)+=(0,0.12182182205939895)};
\addlegendentry{\acs{getp}};
\addplot[blue,dotted,forget plot] coordinates {(2.0, 45.27825) (32.0, 2.829890625)};
\addplot[blue,mark=o,dashed,mark options={solid},error bars/error bar style={solid}] coordinates {(2.0,32.3485)-=(0,0.25091561064189705)+=(0,0.25091561064189705) (4.0,23.601)-=(0,0.4579374708449109)+=(0,0.4579374708449109) (8.0,21.00625)-=(0,0.38287576286428404)+=(0,0.38287576286428404) (16.0,20.29875)-=(0,0.23759862517790964)+=(0,0.23759862517790964)};
\addlegendentry{\acs{getp} - diag. dom.};
\addplot[blue,dotted,forget plot] coordinates {(2.0, 32.3485) (32.0, 2.02178125)};
\addplot[orange!80!black,mark=diamond*] coordinates {(1.0,52.98425)-=(0,1.1695399945182459)+=(0,1.1695399945182459) (2.0,28.919)-=(0,0.8030595370631595)+=(0,0.8030595370631595) (4.0,19.29175)-=(0,0.39098024054730374)+=(0,0.39098024054730374) (8.0,11.8505)-=(0,0.26692486859723985)+=(0,0.26692486859723985) (16.0,10.116)-=(0,0.22057973121238916)+=(0,0.22057973121238916)};
\addlegendentry{\acs{rbt} (w/ IR)};
\addplot[orange!80!black,dotted,forget plot] coordinates {(1.0, 52.98425) (32.0, 1.6557578125)};
\addplot[orange!80!black,mark=diamond] coordinates {(1.0,47.324)-=(0,0.5580437880463904)+=(0,0.5580437880463904) (2.0,25.47075)-=(0,0.2724057890596292)+=(0,0.2724057890596292) (4.0,17.1595)-=(0,0.2286736504811877)+=(0,0.2286736504811877) (8.0,10.57425)-=(0,0.3478683224150032)+=(0,0.3478683224150032) (16.0,9.275500000000001)-=(0,0.08898205658662661)+=(0,0.08898205658662661)};
\addlegendentry{\acs{rbt} (w/out IR)};
\addplot[orange!80!black,dotted,forget plot] coordinates {(1.0, 47.324) (32.0, 1.478875)};
\addplot[green!90!black,mark=triangle*] coordinates {(1.0,34.38975)-=(0,0.33008782391188163)+=(0,0.33008782391188163) (2.0,19.9385)-=(0,0.32598955012700515)+=(0,0.32598955012700515) (4.0,12.40025)-=(0,0.6607002888811753)+=(0,0.6607002888811753) (8.0,7.76475)-=(0,0.07366553789981811)+=(0,0.07366553789981811) (16.0,8.041)-=(0,0.07850900590549159)+=(0,0.07850900590549159)};
\addlegendentry{\acs{genp} (w/ IR)};
\addplot[green!90!black,dotted,forget plot] coordinates {(1.0, 34.38975) (32.0, 1.0746796875)};
\addplot[green!90!black, mark=triangle] coordinates {(1.0,28.21475)-=(0,0.537799298647748)+=(0,0.537799298647748) (2.0,16.572)-=(0,0.12365174178152216)+=(0,0.12365174178152216) (4.0,10.53925)-=(0,0.5322196415697285)+=(0,0.5322196415697285) (8.0,6.608750000000001)-=(0,0.10614222241211202)+=(0,0.10614222241211202) (16.0,7.2177500000000006)-=(0,0.09325278921387614)+=(0,0.09325278921387614)};
\addlegendentry{\acs{genp} (w/out IR)};
\addplot[green!90!black,dotted,forget plot] coordinates {(1.0, 28.21475) (32.0, 0.8817109375)};
\end{loglogaxis} \end{tikzpicture}
	\caption{Strong scaling of various solvers for a problem of size \(n=\num{100000}\).  The dotted lines show ideal scaling.  Error bars show \SI{99}{\%} confidence intervals.}
	\Description{Both \ac{genp} and the \ac{rbt} solver show consistent speedups going from one to eight nodes, although not at the ideal scaling line.  Going to sixteen nodes, \ac{genp} gets a slight slow down and the \ac{rbt} solver only slightly improves.  \Ac{gepp} with a random matrix loses performance going from one to two nodes, but increasing the number of nodes further improves performance at a less than perfect scaling.  (\Ac{gepp} is faster on four nodes than one node by about \(1.3\times\)).  With a diagonally dominant matrix, \ac{gepp} gets small speedups going from one to two nodes and from four to eight nodes.  \Ac{getp} gets some speedup going from two to four nodes, particularly for the non-diagonally-dominant matrix, but little improvement beyond that.}
	\label{fig:exp:perf:strong-100000}
\end{figure}
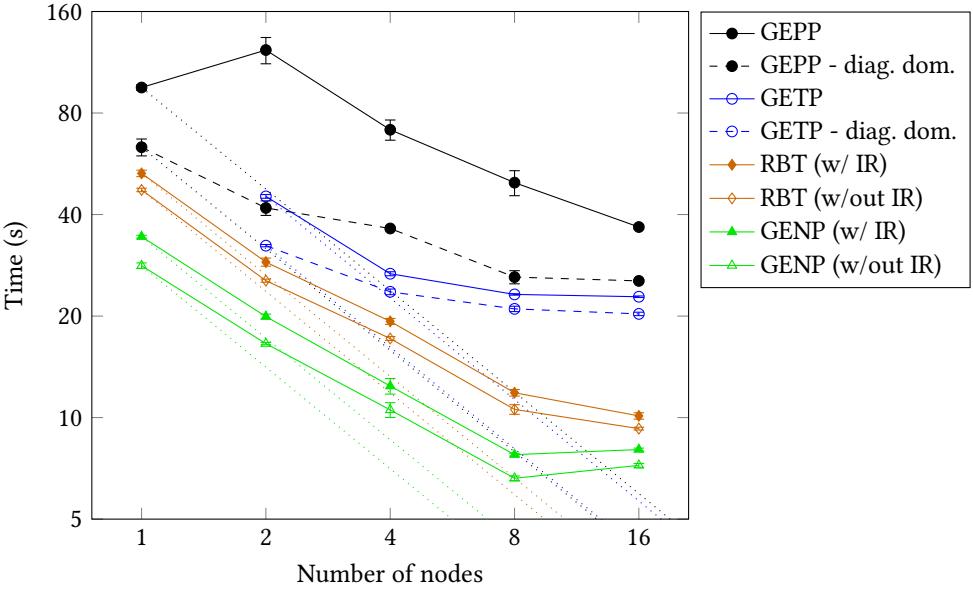
Grids were either \(p\times p\) or \(p\times 2p\), depending on the number of processes.
(Recall that we used 2 processes per node due to the low inter-socket bandwidth.)
Because \ac{slate}'s \ac{getp} allocates extra device memory for cuSOLVER to factor the local panels, it runs out of memory for the single node run.

Overall, the non-pivoted factorizations scale better than the pivoted ones, with \ac{gepp} showing a slowdown when going from one to two nodes.
Note that \ac{gepp} used \(1\times2\) and \(2\times2\) process grids for one and two nodes, respectively, so the pivot search of the latter case has an extra MPI reduction for each column.
A major factor in the worse scaling is that pivoting introduces data dependencies
that reduce the available parallelism.
For example, particularly, the block column (i.e., the panel) must be updated before the diagonal block can be applied to the block row;
on the other hand, \ac{genp} can do these updates simultaneously.
Interestingly, intuition says that increasing the number of processes should increase the number of inter-node row swaps, but \cref{fig:exp:perf:strong-100000} shows worse scaling for the diagonally dominant matrices than for the pivoted matrices, suggesting that overhead for swapping rows parallelizes well.

Interestingly, the ratio of the performance results of the \ac{rbt}
solver and \ac{genp} remains approximately constant from one to eight nodes, but going to sixteen nodes improves the performance of the former while reducing that of the latter.
This suggests that applying the \acp{rbt} scales well enough to partially offset the slowdown in the factorization.

\section{Conclusions}

We have developed a generalized version of the \acl{rbt} that can be
adapted to a matrix size and distribution, instead of needing to adapt
the matrix to the solver.
This formulation achieves the same theoretical result as Parker's original formulation
and similar experimental accuracy at a lower cost.
The key observation was that the \ac{rbt} is simply a tool to randomize the matrix
and that it should be modified to fit the application instead of modifying the application to fit the \ac{rbt}.
This contrasts with the previous work (including our
own~\cite{lindquistReplacingPivotingDistributed2020}) which took the
restriction on matrix size as a given and either ignored it without
comment or suggested users enlarge their matrix to fit the \ac{rbt}.

An important observation from \cref{sec:exp:orr} is that the matrices for which the \ac{rbt}-solver struggles to solve accurately are problems that have large regions with low rank, either originally (as in \matname{ris} and \matname{orthog}) or after applying the \ac{rbt} (as in \matname{riemann}).
This suggests that future research to improve the robustness of \ac{rbt}
preconditioning should consider adding an unstructured reordering to the matrices.
It has previously been suggested to combine \ac{rbt} preprocessing with a technique like \ac{beam} that fixes deficient diagonal elements with local information~\cite{lindquistUsingAdditiveModifications2023}.
However, the observations here would suggest that local corrections will not provide a significant benefit.
(In the case of \ac{beam}, we expect that the additive perturbation would still provide a benefit on top of the \ac{rbt}.)



There are several areas where this work can be extended.
First, as noted in \cref{sec:exp}, the advantages of our generalized formulation still assume a uniform tile size.
While we are unaware of applications that use non-uniform tile sizes with pivoted LU factorization, a further generalization could be explored where the columns and rows of the individual butterfly matrices from \cref{eq:O-matrix-definition} are permuted.
Our implementation supports such a generalization, but we have not investigated its use.
Second, our current \ac{rbt} kernel only runs on CPU.
However, some recent systems, such as the Frontier supercomputer, have the network cards attached to the GPUs; for such systems, doing the \ac{rbt} on the GPUs would reduce data movement.
Furthermore, if the \ac{rbt} size was chosen to eliminate all
inter-node communication, the overhead to transfer remote data to GPU
is removed, making a GPU-based \ac{rbt} more effective.
Thus, it would be valuable to design and test a GPU version of our generalized \ac{rbt}.

In addition to their use in solving systems of linear equations,
butterfly transforms are used in neural networks as a replacement for pointwise convolutions~\cite{alizadehvahidButterflyTransformEfficient2020}, dense layers~\cite{daoLearningFastAlgorithms2019}, and attention mechanisms~\cite{fanAdaptableButterflyAccelerator2022}.
However, the butterfly structure constrains the architecture of networks using those layers.
Furthermore, the butterfly structure is used only to provide a cheap all-to-all transform, which our generalized formulation still provides.
So, our generalized butterfly structure would allow more flexibility in the network architectures.

Finally, many other randomized linear algebra algorithms use random rectangular ``sketching'' matrices, such as those for least squares problems and low-rank approximations~\cite{martinssonRandomizedNumericalLinear2020}.
One such family of random matrices is \acfp{srft}, which randomly sample the rows of an \ac{fft} or \ac{fct} matrix after modifying the columns with random signs and permutations.
However, they can be more expensive than other sparse random matrices when the sketched dimension is large~\cite[Sec.~2.5]{murrayRandomizedNumericalLinear2023}.
However, the experimental success of a depth-2 \ac{rbt} for most matrices suggests that it sketches the matrix such that each \(k\times k\) leading principal submatrix has rank \(k\).
Furthermore, there has been basic experimental success with one-sided \acp{rbt}~\cite{baboulinUsingRandomButterfly2015}, which is even closer to the action of a sketching matrix.
Thus, it is worth investigating an \ac{srft} or \ac{rbt} sketching matrix that uses a depth smaller than \(\log(n)\) and that takes the leading rows instead of randomly subsampling.
The latter simplification could also reduce the computational cost of sketching an \(m\times n\) matrix to \(\Oh{mn}\) because half of the rows could be discarded after each step.
The theoretical analysis of such a transform may not be as strong as a traditional \acp{srft}, but the performance benefits may outweigh that.

\section*{Acknowledgments}

This research was funded in part by the National Science Foundation Office of Advanced Cyberinfrastructure under Grant No.~2004541.
This research was also supported by the Exascale Computing Project,
a collaborative effort of the U.S. Department of Energy Office of Science
and the National Nuclear Security Administration.
Finally, this research used resources of the Oak Ridge Leadership Computing Facility at the Oak Ridge National Laboratory, which is supported by the Office of Science of the U.S. Department of Energy under Contract No.~DE-AC05-00OR22725.

\bibliographystyle{ACM-Reference-Format}
\bibliography{dense-factorization}

\end{document}